\documentclass[10pt,a4paper]{amsart}
\pdfoutput=1 
\usepackage[
text={440pt,575pt},
headheight=9pt,
centering
]{geometry}
\usepackage[utf8]{inputenc}
\usepackage[T1]{fontenc}
\usepackage{amsmath}
\usepackage{amsfonts}
\usepackage{amssymb}
\usepackage{amsthm}
\usepackage{graphicx}
\usepackage{tikz}
\usetikzlibrary{calc}
\usepackage{caption}
\usepackage{subcaption}
\usepackage{natbib}
\usepackage{mathtools}

\usepackage{nicematrix}
\NiceMatrixOptions{cell-space-limits = 1pt}
\usepackage{xcolor}
\usepackage{hyperref}
\usepackage{url, hypcap}
\definecolor{darkblue}{RGB}{0,0,160}
\hypersetup{
	colorlinks,%
	citecolor=darkblue,%
	filecolor=black,%
	linkcolor=darkblue,%
	urlcolor=darkblue
}

\newcommand{\RR}{\mathbb{R}}
\newcommand{\NN}{\mathbb{N}}
\DeclareMathOperator*{\argmax}{\operatorname{arg max}}

\newcommand{\dla}{\langle\!\langle}
\newcommand{\dra}{\rangle\!\rangle}
\newcommand{\B}{\mathbf{B}}
\newcommand{\Q}{\mathbf{Q}}

\newcommand{\bs}{\boldsymbol}

\newtheorem{theorem}{Theorem}
\newtheorem{lemma}[theorem]{Lemma}

\newtheorem{cor}[theorem]{Corollary}
\newtheorem{prop}[theorem]{Proposition}
\newtheorem{conj}[theorem]{Conjecture}

\theoremstyle{definition}
\newtheorem{example}[theorem]{Example}
\newtheorem{remark}[theorem]{Remark}
\newtheorem{defn}[theorem]{Definition}

\newtheorem{algorithm}{Algorithm}

\begin{document}
	\title{Optimal designs for discrete choice models via graph Laplacians}
\author[F.~Röttger]{Frank Röttger}
\address[]{TU Eindhoven\\5600 MB Eindhoven \\The Netherlands}
	\email{f.rottger@tue.nl}
\urladdr{\url{https://sites.google.com/view/roettger/}}
\author[T.~Kahle]{Thomas Kahle}
\address[]{Fakultät für Mathematik\\Otto-von-Guericke Universität
	Magdeburg\\39106 Magdeburg\\Germany}
\email{thomas.kahle@ovgu.de}
\urladdr{\url{https://www.thomas-kahle.de}}
\author[R.~Schwabe]{Rainer Schwabe}
\address[]{Fakultät für Mathematik\\Otto-von-Guericke Universität
	Magdeburg\\39106 Magdeburg\\Germany}
\email{rainer.schwabe@ovgu.de}
\urladdr{\url{https://www.imst3.ovgu.de}}
\date{\today}

\makeatletter
\@namedef{subjclassname@2020}{\textup{2020} Mathematics Subject Classification}
\makeatother


\begin{abstract}
	In discrete choice experiments, the information matrix depends on the model parameters. 
Therefore designing optimally informative experiments for arbitrary initial parameters often yields highly nonlinear optimization problems and makes optimal design infeasible.  To overcome such challenges, we connect design theory for discrete choice experiments with Laplacian matrices of undirected graphs, resulting in complexity reduction and feasibility of optimal design.
We rewrite the $ D $-optimality criterion in terms of Laplacians via Kirchhoff's matrix tree theorem, and show that its dual has a simple description via the Cayley--Menger determinant of the Farris transform of the Laplacian matrix.
This results in a drastic reduction of complexity and allows us to implement a gradient descent algorithm to find locally $ D $-optimal designs.
For the subclass of Bradley--Terry paired comparison models, we find a direct link to maximum likelihood estimation for Laplacian-constrained Gaussian graphical models. 
Finally, we study the performance of our algorithm and demonstrate its application to real and simulated data.
\end{abstract}
\maketitle

\section{Introduction}
Since their Nobel-laureated introduction by \citet{McF74}, discrete choice experiments have enjoyed tremendous success in many applications, for example economics, psychology, public health or transportation.
In this paper we consider discrete choice models with $ m $ unstructured alternatives, i.e.~the alternatives are considered as $m$ categories of a single factor like in a univariate one-way layout.  
A choice set of size $ k $ is a collection of $ k $ out of $ m $ mutually different alternatives which corresponds to blocks in a one-way layout. 
We assume that each alternative from a choice set has a (latent) random utility. 
The consumer then decides in favor of that alternative which has the highest utility within a choice set. 
Hence, in contrast to the standard situation of a one-way layout, the utilities cannot be observed directly, but only which alternative has the highest utility within a choice set.
This results in choice probabilities which depend on the mean utilities.
Thereby the model can be reformulated as a multinomial regression model.
The model with $ k=2 $ alternatives per choice set is known as the Bradley--Terry paired comparison model \citep{BradleyTerry}, which had already been introduced by \citet{zermelo29} to estimate the playing strength of chess players in tournaments.

The quality of the outcome of a discrete choice experiment strongly depends on the assignment of the alternatives among different choice sets, that is the experimental design.
In this paper our interest is in $ D $-optimal designs for such models, that is designs which maximize the determinant of the information matrix of the experiment.
Due to the inherent nonlinearity of multinomial regression, $ D $-optimal designs are only locally optimal with respect to the parameter of the model \citep{Chernoff}.
This parameter dependence makes the optimization problem much more complicated, for example in comparison with optimal designs for linear models, where there is no such parameter dependence.
Therefore, finding $D$-optimal designs given arbitrary parameters in generalized linear models was considered impossible in practice \citep{Dette1996}.
As a consequence, previous research is limited to very restricted models.
For example, \citet{Grasshoff2008} and \citet{KRS2021} study $D$-optimal designs for the special setting of Bradley--Terry paired comparison models with only $ m\in\{3,4\} $ alternatives, and describe the geometry of saturated designs, that is designs with minimal support.

An alternative approach to ranking is owner assisted scoring,
for example as described by \citet{SuScoring} or
\citet{frongillo2021general}.  This however works under different
premises.  There an owner of the items knows the true utilities and
should be incentivised to improve upon a noisy observation of
utilities.  Discrete choice models work under the premise that the
utilities are not directly observable, but only (partial) rankings can
ever be observed.

As our mathematical contribution, we connect design theory
for discrete choice experiments with Laplacian matrices and
Laplacian-constrained Gaussian graphical models.
The Laplacian matrix
of a weighted graph is the difference of a diagonal degree matrix and
the edge weight matrix (see Section~\ref{s:preliminaries} for the
formal definition).  A Laplacian-constrained Gaussian graphical model
is defined as a multivariate Gaussian with a Laplacian matrix as
inverse covariance.  These models are an active research area in
machine learning, see e.g.~\citet{EPO2017}, \citet{KYVP2019} or
\citet{YCP2021}, and multivariate extremes~\citep{REZ21}.  Laplacian
matrices are very relevant in combinatorics and graph theory and
connect geometric, graphical and algebraic properties, see for example
\citet{devriendt2020} for a concise introduction.

The relevance of Laplacian matrices for design theory originates in a reformulation of the $ D $-optimality criterion via Kirchhoff's matrix tree theorem.
Similar connections between Laplacian matrices and optimal designs are well known in the design literature, see for example \cite{BC2009} for a survey article.
Each Laplacian matrix is one-to-one with a variogram or Euclidean distance matrix $ \Gamma $.
Our main result is a dual description of the $ D $-criterion as maximizing the logarithmic Cayley--Menger determinant
\begin{align}
	\argmax_{\Gamma}\quad	\log \det \begin{pmatrix}
		0&-\mathbf{1}^T\\
		\mathbf{1}&-\frac{\Gamma}{2}\\
	\end{pmatrix}\quad \text{subject to } A\vec{\Gamma}\le (m-1)\bs 1\label{eq:CM}
\end{align}
under simple linear inequality constraints defined by a parameter-dependent $ \binom{m}{k}\times\binom{m}{2}$-matrix~$A$.  See Proposition~\ref{prop:dual} and Theorem~\ref{thm:KKT-conditions} for details and notation.
There are two simplifications in this formulation.  Both are very relevant for high-dimensional problems and yield simple algorithms to find $ D $-optimal designs:
\begin{enumerate}
	\item[1.] While the original optimization problem is defined via $ \binom{m}{k} $ variables, the dual problem is written in $ \binom{m}{2} $ variables. This results in a drastic complexity reduction for $ k>2 $.
	\item[2.] The inequalities $A\vec{\Gamma} \le (m-1)\bs 1$ are sparse.  Each inequality has $\binom{k}{2}$ terms and thus the system allows for efficient algorithmic handling.
\end{enumerate}

The optimization problem \eqref{eq:CM} has appealing properties for high-dimensional settings, as we show by a connection to maximum likelihood estimation for Laplacian-constrained Gaussian graphical models.  In the case of the Bradley--Terry paired comparisons this link even allows for symbolic solutions
for $ D $-optimal designs
(Theorem~\ref{thm:decomposable}).
This follows because the maximum likelihood estimator of a decomposable Laplacian-constrained Gaussian graphical model has a rational description as a matrix completion problem.
We further illustrate Theorem~\ref{thm:decomposable} with examples in
the appendix.

The particularly nice structure of \eqref{eq:CM} allows for a gradient descent algorithm that we introduce in Section~\ref{sec:alg}. 
	We study the algorithms performance and observe in simulations that it finds the solution of \eqref{eq:CM} very quickly and with good precision.
We further apply our algorithm to two data sets from the package \texttt{hyper2} \citep{Hankin2017} and simulated data in Section~\ref{sec:applications}.
Our new methodology allows us to quickly compute a $ D $-optimal design with respect to the estimated parameter and to evaluate the $ D $-efficiency of the study designs.
We find that in both real data sets, the study designs have a lower $ D $-efficiency than the complete design with constant design weights.
Furthermore, we observe that given the comparably similar choice probabilities, both the study designs and the complete designs with constant design weights exhibit a measurable decrease in $ D $-efficiency.
In a simulation study, we sample parameters from a log-normal distribution for growing standard deviation.
It shows that on average both the $ D $-efficiency of the complete design with constant weights and the support size of the $ D $-optimal design decrease.
Finally, we study the $ D $-efficiency of the complete design with constant design weights and two complementary balanced incomplete block designs on a line in parameter space.
All \texttt{R}- and \texttt{Mathematica} code for the simulations, applications as well as intricate matrix operations is available as a GitHub repository:
{\small \url{https://github.com/frank-unige/discrete_choice_designs_via_graph_Laplacians}}

\section{Preliminaries}\label{s:preliminaries}
\subsection{Model}
We study discrete choice models for $ m $ unstructured alternatives. 
A choice set $ C_j\subset [m]:=\{1,\ldots,m\} $ is a subset of alternatives.
There are $ \binom{m}{k} $ different choice sets with exactly $k$ alternatives, such that $ j=1,\ldots,\binom{m}{k} $.
We assume that all choice sets in an experiment consist of $k$ alternatives.
For each choice set $ C_j $, let $ Y(C_j) $ be a $ k $-variate random vector consisting of binary components $ Y(i,C_j) $ for $ i\in C_j $.
Here, $ Y(i,C_j)=1 $ means that $ i $ is the preferred alternative in choice set $ C_j $ and $ Y(i,C_j)=0 $ that this is not the case.
It holds that $\sum_{i\in C_j}Y(i,C_j)=1$, such that $Y(C_j)$ follows a multinomial distribution.
For each $i$, we have a parameter $ \pi_i>0,~i\in [m] $, the inherent, latent attractiveness (mean utility) of alternative~$i$.
The model is specified by
\begin{align}
	\mathbb{P}(Y(i,C_j)=1)&=\frac{\pi_i}{\sum_{s\in C_j}\pi_s}.\label{eq:model}
\end{align}
For $ k=2 $ one has only paired comparisons and this is the Bradley--Terry model, see Section~\ref{sec:BradleyTerry}.
The vector $ Y(C_j)=(Y(i,C_j))_{i\in C_j} $ has a multinomial distribution with success probabilities~\eqref{eq:model}.
The parameters $ \pi $ are not identifiable, as multiplying the vector $ \pi $ with a scalar factor does not change the probabilities in~\eqref{eq:model}.
Therefore, a standard approach is to reduce the model to $ m-1 $ parameters, for example by fixing~$ \pi_m $.
As this renders the remaining parameters identifiable, such a reduction is called an identifiability condition (see Section~\ref{sec:identifiability}).
After a logarithmic transformation $ \beta_{i}=\log(\pi_i) $, the probability \eqref{eq:model} becomes
\begin{align}
	\mathbb{P}(Y(i,C_j)=1)&=\frac{\exp(\beta_i)}{\sum_{s\in C_j}\exp(\beta_s)}. \label{eq:model_beta}
\end{align}
This transformation allows for a reformulation as a generalized linear model, where the mean response $ \mathbb{E}(Y(i,C_j)) $ is linked to the linear predictor $ f(i)^T\beta = \beta_i $ with parameter vector $ \beta\in \RR^{m} $ and regression vectors $ f\colon[m]\to \RR^m,\, f(i)=e_i $, where $ e_i $ denotes the $ i $-th canonical unit vector, by a logit transformation.
Hence, we have a multinomial regression model
\begin{equation*}
	\mathbb{E}(Y(i,C_j))=\mathbb{P}(Y(i,C_j)=1) = 
	\frac{\exp\left(f(i)^T\beta\right)}
	{\sum_{s\in C_j}\exp\left(f(s)^T\beta\right)}.
\end{equation*}

\subsection{Graph Laplacians and the Farris transform}
We give a short introduction to graph Laplacians and the Farris transform. For an extensive treatment see e.g.~\citet{devriendt2020}.
Let $ G=(V,E) $ be a simple undirected graph with vertex set $ V=\{1,\ldots,m\} $ and edge set $ E\subseteq V\times V $.
Let $ W_{uv} $ denote positive edge weights on $ G $, associated with each edge $ uv\in E $, where $uv$ or $u,v$ denotes an edge formed by vertices $u$ and $v$. We say that non-edges have weight $ 0 $.
The Laplacian matrix $ L $ of a weighted graph is defined as follows: 
\begin{align*}
	L_{uv}&=\begin{cases}
		-W_{uv},\quad u\neq v \\
		\sum_{\ell\in \text{adj}(u)} W_{u\ell},\quad u=v,\\ 
	\end{cases}
\end{align*}
where $ \text{adj}(u) $ denotes the set of adjacent vertices of $u$.
Then $L$ is a symmetric matrix with non-positive non-diagonal entries,
a symmetric Z-matrix.  Additionally the row sums are zero.
\begin{example}
	Given the weighted graph in Figure~\ref{fig:graph_ex} with positive edge weights $ W_{uv},\; uv\in E $.
	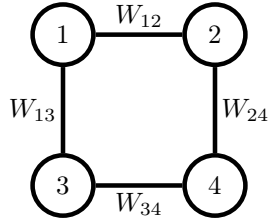
\begin{figure}\centering \begin{tikzpicture}[line width=.6mm]
			\node[minimum size=8mm,shape=circle,draw=black] (1) at (-1,1) {1};
			\node[minimum size=8mm,shape=circle,draw=black] (2) at (1,1) {2};
			\node[minimum size=8mm,shape=circle,draw=black] (3) at (-1,-1) {3};
			\node[minimum size=8mm,shape=circle,draw=black] (4) at (1,-1) {4};
			\path (1) edge (2);
			\path (1) edge (3);
			\path (2) edge (4);
			\path (3) edge (4); 
			\node (5) at (0,1.25) {$W_{12}$};
			\node (6) at (-1.4,0) {$W_{13}$};
			\node (7) at (1.4,0) {$ W_{24} $};
			\node (8) at (0,-1.25) {$W_{34}$};
		\end{tikzpicture}
		\caption{Weighted graph}\label{fig:graph_ex}
	\end{figure}
	The Laplacian matrix of this graph equals
	\[	L= \begin{pmatrix}
		W_{12}+W_{13}&-W_{12}&-W_{13}&0\\
		-W_{12}&W_{12}+W_{24}&0&-W_{24}\\
		-W_{13}&0&W_{13}+W_{34}&-W_{34}\\
		0&-W_{24}&-W_{34}&W_{24}+W_{34}\\
	\end{pmatrix}. \] 
	Note that $ L\mathbf{1}=\mathbf{0} $, i.e.~that the row- and column-sums vanish.
\end{example}

\begin{remark}
	For unweighted graphs, the combinatorial Laplacian matrix is the difference of the diagonal degree matrix $ D $ where $ D_{uu}=|\text{adj}(u)| $ and the adjacency matrix. This is equivalent to $ W_{uv}=1 $ for all $ uv\in E $ in the definition above.
\end{remark}
Let $ \mathbb{S}^{m-1} $ be the set of symmetric $ (m-1)\times (m-1) $-matrices.
The Farris transform of a matrix $ A\in \mathbb{S}^{m-1} $ is a linear transformation resulting in the $m\times m$ matrix $\Gamma$ with entries
\begin{align*}
	\begin{cases}
		\Gamma_{uv}= A_{uu}+A_{vv}-2A_{uv}, & u,v<m,\\
		\Gamma_{um}=\Gamma_{mu}= A_{uu}, & u<m,\\
		\Gamma_{mm}=0. &\\
	\end{cases}
\end{align*}
The matrix $ \Gamma$ lies in $ \mathbb{S}_0^m $, the set of symmetric $ m\times m $ matrices with zero diagonal.
Furthermore, the matrix $ A $ is positive definite if an only if $ \Gamma\in \mathcal{C}^m $, where $ \mathcal{C}^m\subset \mathbb{S}_0^m $ is the cone of conditionally negative definite symmetric $ m\times m $ matrices \citep{EH2020}.
The inverse Farris transform reconstructs the entries of $A$ via
\begin{align}
	A_{uv}= \frac{1}{2}(\Gamma_{um}+\Gamma_{vm}-\Gamma_{uv}). \label{eq:invFarris}
\end{align}
Let $ \mathbb{S}_{\ge}^m $ be the cone of symmetric positive semidefinite $ m\times m $ matrices.
We call $ \mathbb{U}^{m}:=\{B\in \mathbb{S}_{\ge}^m : B\mathbf{1}=0\} $ the set of symmetric $ m\times m $ matrices with row sums equal to zero. 
Let $ \Theta\in\mathbb{U}^{m} $. 
As the diagonal entries of $ \Theta $ equal the negative of the sum of the respective non-diagonal entries in each row and column, the entries of $ \Theta $ are uniquely characterized by any $ (m-1)\times (m-1) $ principal submatrix $ \Theta^{(k)} $, resulting from deleting the $ k $-th row and column of~$\Theta$.
We obtain a bijection between $ \mathcal{C}^m $ and $ \mathbb{U}^{m} $ via the Farris transform of the inverse of~$ \Theta^{(k)} $.

Let $ Q_{uv}=-\Theta_{uv} $ for all $ u\neq v $ and $ Q_{uu}=0 $ for $ 1\le u \le d $. When $ Q_{uv}\ge 0 $, then $ \Theta $ is a graph Laplacian.
Then $ Q \in \mathbb{S}_0^m $ is a symmetric matrix with zero diagonal.
The Farris transform relates to inner products as follows.
Let $\langle A, B \rangle = \text{tr}(AB)$ be the standard trace inner product on $\mathbb{S}^{m-1}$ and $ \dla \Gamma,Q \dra:=\sum_{s<t}\Gamma_{st}Q_{st} $ the vector  inner product on $ \mathbb{S}_0^m $.
For arbitrary $ A,\Theta^{(k)}\in\mathbb{S}^{m-1} $ and $\Gamma$ the Farris transform of $A$, it holds that
\begin{align}
	\langle A,\Theta^{(k)} \rangle = \dla \Gamma, Q \dra.\label{eq:inner_product}
\end{align}

\subsection{Information matrix}\label{sec:informationM}
Let $ C_j\subset V $ be a choice set of size $ k $. We define
$ \Lambda_j(\pi) $ as the Laplacian matrix of the complete graph
$ \mathcal{K}_{C_j} $ on the vertex set $ C_j $ with edge weights
$\pi_s\pi_t(\sum_{i\in C_j}\pi_i)^{-2} $ for $ s,t\in C_j $.  Under the assumption of the model \eqref{eq:model}, the $ k\times k $ covariance matrix of $ Y(C_j) $ equals
$\text{Cov}(Y(C_j))=\Lambda_j(\pi)$ (as noted by \citet{GGHS2013}).
It follows that the information matrix for one observation
of $ Y(C_j) $ computes as
\[M(C_j,\pi)=F(C_j)^T \text{Cov}(Y(C_j))F(C_j)=F(C_j)^T\Lambda_j(\pi) F(C_j), \]
where $ F(C_j)=(f(i))_{i\in C_j}^T $ denotes the $ k\times m $ design matrix of $ C_j $ \citep[Eq.~(22.3)]{GS2015}.
Then, $ M(C_j,\pi) $ is the Laplacian matrix of the graph $ (V,E(\mathcal{K}_{C_j})) $ resulting from adding the nodes $ V\setminus C_j $ (but no edges) to $ \mathcal{K}_{C_j} $.

\subsection{Design}
An experimental design for the model \eqref{eq:model} assigns proportions of the total number of observations $ N $ to the different choice sets.
These proportions are called \textit{design weights}.
The design is encoded as a $ \binom{m}{k} $-dimensional vector $ \xi $ with non-negative entries that sum up to one.
To simplify computations, one allows for the design weights to live in $ \RR $ instead of $ \NN/N $.
In such a case, in the spirit of~\citet{Kiefer1959,Kief74}, one speaks of an \textit{approximate} design.
We denote the set of all approximate designs for $ \binom{m}{k} $ choice sets as
\[\Delta_{\binom{m}{k}}:=\{\xi \in \RR^{\binom{m}{k}}_{\ge 0} : \Vert \xi \Vert_1=1 \}. \]
Let $ \xi=(w_1,\ldots,w_{\binom{m}{k}}) $ be an approximate design, where $w_j$ is the weight of choice set $C_j$. 
Assuming independent observations, we define the information matrix of $ \xi $ as the convex combination of the choice set information matrices:
\[M(\xi,\pi)=\sum_{j=1}^{\binom{m}{k}}w_j M(C_j,\pi). \]
If the design weights are in $ \NN/N $, the matrix $ N\cdot M(\xi,\pi) $ equals the classical information matrix of $ N $ independent observations taken according to~$\xi$.
As a convex combination of graph Laplacians, the information matrix $ M(\xi,\pi) $ is itself a graph Laplacian.
The edge weights are $ \pi_s\pi_t\sum_{j:s,t\in C_j}w_j/(\sum_{i\in C_j}\pi_i)^{2} $, see also \citet[p.~146]{SD2016}, thus
\begin{align}
	M_{st}(\xi,\beta)&=\begin{cases}
		-\pi_s\pi_t\sum_{j:s,t\in C_j}\frac{w_j}{(\sum_{i\in C_j}\pi_i)^{2}}, & s\neq t,\\
		\sum_{u\neq s}\pi_s\pi_u\sum_{j:s,u\in C_j}\frac{w_j}{(\sum_{i\in C_j}\pi_i)^{2}}, & s=t.\\
	\end{cases}\label{eq:gen_information}
\end{align}

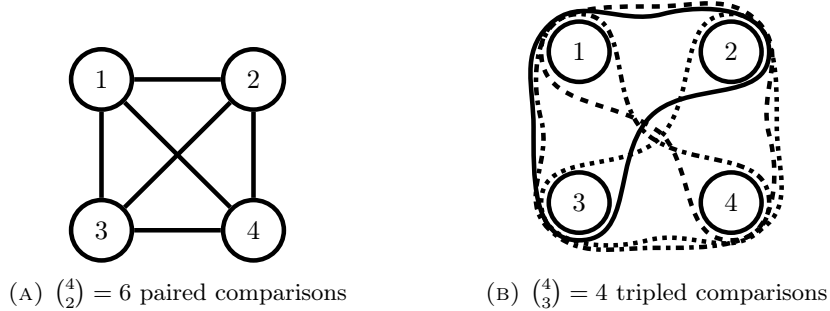
\begin{figure}\centering
	\begin{subfigure}[b]{0.4\textwidth}
		\centering
		\begin{tikzpicture}[line width=.6mm]
			\node[minimum size=8mm,shape=circle,draw=black] (1) at (-1,1) {1};
			\node[minimum size=8mm,shape=circle,draw=black] (2) at (1,1) {2};
			\node[minimum size=8mm,shape=circle,draw=black] (3) at (-1,-1) {3};
			\node[minimum size=8mm,shape=circle,draw=black] (4) at (1,-1) {4};
			\path (1) edge (2);
			\path (1) edge (3);
			\path (1) edge (4);
			\path (2) edge (3);
			\path (2) edge (4);
			\path (3) edge (4); 
		\end{tikzpicture}
		\caption{$ \binom{4}{2} =6 $ paired comparisons}
	\end{subfigure}
	\begin{subfigure}[b]{0.4\textwidth}
		\centering
		\begin{tikzpicture}[line width=.6mm]
			\node[minimum size=8mm,shape=circle,draw=black] (1) at (-1,1) {1};
			\node[minimum size=8mm,shape=circle,draw=black] (2) at (1,1) {2};
			\node[minimum size=8mm,shape=circle,draw=black] (3) at (-1,-1) {3};
			\node[minimum size=8mm,shape=circle,draw=black] (4) at (1,-1) {4};
			\begin{scope}
				\draw  ($(1) + (-0.4, 0.35)$)
				to [out = 45, in = 195] ($(0.1, 1.5)$)
				to [out = 15, in = 90] ($(2) + (0.5, 0)$)
				to [out = 270, in = 50] ($(-0.1, 0.1)$)
				to [out = 230, in = 0] ($(3) + (0, -0.5)$)
				to [out = 180, in = 270] ($(-1.6,0)$)
				to [out = 90, in = 225] ($(1) + (-0.4, 0.35)$);
			\end{scope}
			\begin{scope}
				\draw [dashed] ($(2) + (0.35,0.4)$)
				to [out = 315, in = 105] ($(1.5,-0.1)$)
				to [out = 285, in = 0] ($(4) + (0, -0.5)$)
				to [out = 180, in = 310] ($(0.1, -0.1)$)
				to [out = 140, in = 270] ($(1) + (-0.5,0)$)
				to [out = 90, in = 180] ($(0, 1.6)$)
				to [out = 0, in = 135] ($(2) + (0.35,0.4)$);
			\end{scope}
			\begin{scope}
				\draw [dash dot] ($(3) + (-0.4, -0.35)$)
				to [out = 135, in = 285] ($(-1.5,0.1)$)
				to [out = 105, in = 180] ($(1) + (0, 0.5)$)
				to [out = 0, in = 140] ($(-0.1, -0.1)$)
				to [out = 320, in = 90] ($(4) + (0.5,0)$)
				to [out = 270, in = 0] ($(0,-1.6)$)
				to [out = 180, in = 315] ($(3) + (-0.4, -0.35)$);
			\end{scope}
			\begin{scope}
				\draw[dotted] ($(4) - (-0.4, 0.35)$)
				to [out = 225, in = 15] ($(-0.1, -1.5)$)
				to [out = 195, in = 270] ($(3) - (0.5, 0)$)
				to [out = 90, in = 230] ($(0.1, -0.1)$)
				to [out = 50, in = 180] ($(2) - (0, -0.5)$)
				to [out = 0, in = 90] ($(1.6,0)$)
				to [out = 270, in = 35] ($(4) - (-0.4, 0.35)$);
			\end{scope}
		\end{tikzpicture}
		\caption{$ \binom{4}{3} =4 $ tripled comparisons}
	\end{subfigure}
	\caption{Complete hypergraphs for $ m=4 $ and $ k\in \{2,3\} $}\label{fig:graphs}
\end{figure}
\begin{remark}
	A design $ \xi \in \Delta_{\binom{m}{k}} $ also has a representation as a weighted $ k $-uniform hypergraph $ G=(V,H) $. Here, the vertices $ V=[m] $ are the alternatives, and the hyperedges $ H={C_1,\ldots,C_{\binom{m}{k}}} $ are the choice sets.
	The hyperedges are weighted with the corresponding design weights, where a zero weight stands for a non-hyperedge.
	Information matrices become weighted analogues of the hypergraph Laplacian matrices of~\citet{Rod2009}.
	In the Bradley--Terry model, the hypergraphs are ordinary graphs and the design defines the same undirected graph as the information matrix \citep{KRS2021}.
	Small examples for $ k=2 $ and $ k=3 $ are shown in Figure~\ref{fig:graphs}.
\end{remark}

\subsection{Identifiability}\label{sec:identifiability}
Since the information matrices $ M(C_j,\pi) $ and $ M(\xi,\pi) $ are singular (as they are graph Laplacians), the model \eqref{eq:model} is not a-priori identifiable, but it becomes identifiable after fixing one parameter.
For example, we set $ \pi_m=1 $, such that $ \beta_m=0 $.
We write $ \pi_0=(\pi_1/\pi_m,\ldots,\pi_{m-1}/\pi_m)\in \RR_{\ge 0}^{m-1} $ for the parameter of the reduced model where we divide $ \pi $ by~$ \pi_m $.
Then $ M(\xi,\pi)=M(\xi,(\pi_0,1)) $, as normalizing with $ \pi_m=1 $ neither changes \eqref{eq:model}, nor the information matrix \eqref{eq:gen_information}.

A standard approach for discrete choice models uses these identifiability constraints.
For example, in \citet{KRS2021}, it is assumed that $ \pi_m=1 $ and an $ (m-1)\times (m-1) $ information matrix is computed in this reduced model.
This information matrix equals $M^{(m)}(\xi,\pi_0)$, which is obtained from $ M(\xi,\pi) $ by deleting the $m$-th row and column, i.e.
\[M_{st}^{(m)}(\xi,\pi_0)=M_{st}(\xi,\pi),\quad s,t< m.\]
On the other hand, $ M(\xi,\pi) $ is recovered from $ M^{(m)}(\xi,\pi_0) $ by
\begin{align}
	M_{st}(\xi,\pi)&=\begin{cases}
		M_{st}^{(m)}(\xi,\pi_0), &\text{ for } s,t\neq m,\\
		-\sum_{u=1}^{m-1} M_{su}^{(m)}(\xi,\pi_0), &\text{ for } s\neq m, t=m,\\
		\sum_{u=1}^{m-1}	\sum_{\ell=1}^{m-1} M_{u\ell}^{(m)}(\xi,\pi_0), &\text{ for } s=t=m.\\
	\end{cases}\label{eq:information_trafo}
\end{align}

\subsection{Optimal designs}		
The asymptotic covariance of the maximum likelihood estimator for identifiable $\beta$ in generalized linear models is proportional to the inverse of the information matrix \citep{FK1985}. This is the central reason why design theory for generalized linear models aims at maximizing the information of an experiment.
Hereby, one usually chooses a function that maps the information to the real line.
These functions are known as \textit{optimality criteria}.
But unlike in linear models, the information for generalized linear models depends on the parameter~$\pi$.
This means that optimality is local in the parameter space.
Among the most popular criteria is the $ D $-criterion that maximizes the logarithmic determinant of the information:

\begin{defn}
	Consider the criterion $ \phi\colon \RR^{(m-1)\times(m-1)}\to \RR, \phi(\bar{M})=\log \det (\bar{M}) $.   A design $ \xi^* $ is locally $ D $-optimal at parameter $ \pi_0 $ when $ \phi(M^{(m)}(\xi^*,\pi_0))\ge \phi(M^{(m)}(\xi,\pi_0))  $ 
	for all $ \xi \in \Delta_{\binom{m}{k}} $.
\end{defn}
A locally $ D $-optimal design at parameter $ \pi_0 $ is therefore an optimum as follows:
\begin{align}\label{eq:D_optim_problem}
	\xi^*&=\argmax_{\xi\in \Delta_{\binom{m}{k}}}\, \phi(M^{(m)}(\xi,\pi_0)).
\end{align}
This is an optimization problem for which the target function depends on $\pi_0$.
As the information matrix is linear in $\xi $, the function $ \phi $ is concave and the set $\Delta_{\binom{m}{k}}$ is convex, this is a convex optimization problem.
Thus each local optimum (as a function of $\xi$, given $\pi_{0}$) is global.
While each design $\xi\in \Delta_{\binom{m}{k}}$ defines a unique point in the information matrix polytope \[
\bar{\mathcal{M}}:=\text{convhull}\left[M^{(m)}(C_j,\pi_0),1\le j\le \binom{m}{k}\right],
\]
the converse is not true.  
One could therefore also view \eqref{eq:D_optim_problem} as a two-stage problem.
First optimize $\phi$ over the polytope $\bar{\mathcal{M}}$, yielding an optimal information matrix.
This optimal matrix may typically have many expressions as a convex combination of the vertices of $\bar{\mathcal{M}}$ and expressing it as such is picking an optimal design~$\xi^*$.

\section{Optimal designs for discrete choice models}
In this section we rephrase the optimization problem \eqref{eq:D_optim_problem} via graph Laplacians and the Farris transform and find a simple dual of the rephrased problem.
To improve readability, we simplify the notation in this section by leaving out the parameter~$\pi$, though everything is local.

By Section~\ref{sec:informationM}, the information matrix $ M(\xi) $ of a discrete choice design $ \xi=(w_1,\ldots,w_{\binom{m}{k}}) $ is the Laplacian matrix of a graph $ G $ with edge weights 
\begin{align}
	Q_{uv}(\xi)&:=\pi_{u}\pi_{v}\sum_{j:uv\in C_j}\frac{w_j}{(\sum_{i\in C_j}\pi_i )^2}. \label{eq:edge_weights}
\end{align}
By Kirchhoff's matrix tree theorem, we can write the determinant of the reduced information matrix in terms of the graph weights, that is
\begin{align*}
	\det(M^{(m)}(\xi))&=\sum_{T\in\mathcal{T}}\prod_{uv\in T}Q_{uv}(\xi),
\end{align*}
where $ \mathcal{T} $ is the set of all spanning trees of $ G $, see e.g.~\citet[Lemma~4.4]{REZ21}.
As a consequence, the $ D $-criterion rewrites as
\begin{align}
	\text{minimize} \quad& -\log \sum_{T\in\mathcal{T}}\prod_{uv\in T}Q_{uv}(\xi),\quad \text{subject to }\; \xi\in \Delta_{\binom{m}{k}}.\label{eq:opt_problem_DC}
\end{align}
We rewrite $ Q(\xi) $ from \eqref{eq:edge_weights} as a vector $\vec{\Q}(\xi)=(Q_{12}(\xi),Q_{13}(\xi),\ldots,Q_{\binom{m}{2}-1,\binom{m}{2}}(\xi))$.
For this, we use the lexicographic ordering $ (12,13,\ldots,(m-1)m) $ to transform a matrix from $ \mathbb{S}_0^{m} $ to a $ \binom{m}{2} $-variate vector.
Similarly, let $ \vec{\Gamma}(\xi) $ denote the vectorization of $ \Gamma(\xi) $.
The following definition introduces edge-hyperedge incidence vectors and matrices which encode which edges are contained in a hyperedge, or a collection of hyperedges:
\begin{defn}
	For each choice set $ C_j $ let $ s_j $ denote its incidence vector in the space of the edges of the underlying $ m $-simplex, such that the $ uv $-th entry of $ s_j $, where $ s_j $ is indexed in lexicographic ordering, equals
	\[s_{j,uv}=\begin{cases}
		1, & u,v \in C_j,\\
		0, & \text{otherwise}.
	\end{cases}\]
	Let $S$ be the matrix whose rows are the incidence vectors, i.e.~$S^T:=\left(s_1,\ldots,s_{\binom{m}{k}}\right)$.
	Let $ \text{supp}(\xi)=\{C_j:w_j>0\}  $ denote the set of choice sets with non-zero design weight.
	We define $ S(\xi) $ as the edge incidence matrix of all choice sets in the support of $ \xi $, so that $ S(\xi) $ is a $ \vert\text{supp}(\xi)\vert\times \binom{m}{2} $-submatrix of $ S $ containing all incidence vectors for $ C_j \in \text{supp}(\xi) $.
\end{defn}
Next, we define two diagonal matrices
\begin{equation*}
	R:=\text{diag}\left(\left(\sum_{i\in C_1}\pi_i \right)^2,\ldots \left(\sum_{i\in C_{\binom{m}{k}}}\pi_i \right)^2 \right), \qquad
	L:=\text{diag}(\pi_{1}\pi_{2},\ldots,\pi_{m-1}\pi_{m}).
\end{equation*}
These matrices allow us to rewrite the vector $ \vec{\Q}(\xi) $ with respect to $ \xi $ as follows:
\begin{lemma}
	For a discrete choice design $ \xi $, it holds that \label{lem:Q_in_xi}
	\[\vec{\Q}(\xi)=LS^{T}R^{-1}\xi.\]
\end{lemma}

Let $ \Sigma^{(m)}(\xi) $ be the inverse of the reduced information matrix $ M^{(m)}(\xi) $, and let $ \Gamma(\xi) $ be the Farris transform of $ \Sigma^{(m)}(\xi) $.
The dual problem of the $ D $-criterion \eqref{eq:opt_problem_DC} simplified in~$\Gamma(\xi)$.
\begin{prop}\label{prop:dual}
	The dual problem of \eqref{eq:opt_problem_DC} is
	\begin{equation}\label{eq:dual}
		\max_{\Gamma}\;\log \det\begin{pmatrix}
			0&-\mathbf{1}^T\\
			\mathbf{1}&-\frac{\Gamma(\xi)}{2}\\
		\end{pmatrix}, \quad \text{subject to }~~ \Gamma(\xi)\in \mathcal{C}^d \text{ and } R^{-1}S L \vec{\Gamma}(\xi) \le (m-1) \mathbf{1}.
	\end{equation}	
\end{prop}
In the proof of Proposition~\ref{prop:dual} we derived the Karush-Kuhn-Tucker conditions \eqref{eq:KKT1}-\eqref{eq:KKT5}, which certify optimality and therefore allow an equivalent description of $ D $-optimality:
\begin{theorem}\label{thm:KKT-conditions}
	A discrete choice design $ \xi^* $ is $ D $-optimal if and only if all of the following hold:
	\begin{enumerate}
		\item[(i)] $ \xi^*\in\Delta_{\binom{m}{k}} $,
		\item[(ii)] $ R^{-1}S L \vec{\Gamma}(\xi^*)\le (m-1)\mathbf{1} $,
		\item[(iii)] $ \langle R^{-1}S L \vec{\Gamma}(\xi^*)- (m-1)\mathbf{1}, \xi^*\rangle=0 $.
	\end{enumerate}
\end{theorem}
Theorem~\ref{thm:KKT-conditions} (iii) is equivalent to $ (R^{-1}S L \vec{\Gamma}(\xi^*)- (m-1)\mathbf{1})\cdot \xi^*=\mathbf{0} $, where $ \cdot $ denotes entry-wise multiplication.
The conditions in Theorem~\ref{thm:KKT-conditions} are rephrasing classical optimal design results in the Farris transform $ \Gamma $. A discussion on this relation, together with an example for $ m=6 $ and $ k=3 $ is available in the appendix.

\section{Bradley--Terry paired comparison model}\label{sec:BradleyTerry}

For the Bradley--Terry model ($ k=2 $), the edge-hyperedge incidence matrix $ S $ is the identity matrix. 
This leads to the following corollary of Proposition~\ref{prop:dual}.
\begin{cor}~\label{cor:dualBT}
	For the Bradley--Terry paired comparison model, \eqref{eq:dual} simplifies to
	\begin{equation}\label{eq:dual_BT}
		\operatorname{maximize}\quad\log \det\begin{pmatrix}
			0&-\mathbf{1}^T\\
			\mathbf{1}&-\frac{\Gamma(\xi)}{2}\\
		\end{pmatrix},
		\quad
		\text{subject to } \Gamma(\xi) \le \overline{\Gamma},
	\end{equation}
	where $ \vec{\overline{\Gamma}}=(m-1) L^{-1}R\mathbf{1} $.
\end{cor}
This optimization problem is equivalent to the dual optimization problem for Gaussian maximum likelihood estimation under Laplacian constraints, compare for example \citet{REZ21} or \citet{YCP2021}.
In fact, Corollary~\ref{cor:dualBT} equals Proposition~6.2 of \citet{REZ21} with $ \vec{\overline{\Gamma}}=(m-1) L^{-1}R\mathbf{1} $ in matrix form.

We denote the choice set that contains the alternatives $ u $ and $ v $ with $ (u,v) $.
Let $ \lambda_{uv}:= \frac{\pi_u \pi_v}{(\pi_u+\pi_v)^{2}} $.
It follows that $ \overline{\Gamma}_{uv}=\frac{m-1}{\lambda_{uv}} $.
As a consequence of Corollary~\ref{cor:dualBT}, we find the following corollary of Theorem~\ref{thm:KKT-conditions}:
\begin{cor}\label{cor:KKT-conditions}
	A Bradley--Terry design $ \xi^* $ is $ D $-optimal if and only if
	\begin{enumerate}
		\item[(i)] $ \xi^*_j\ge 0 $ for all $ 1\le j \le \binom{m}{2} $, and
		\item[(ii)] $ \Gamma(\xi^*)\le \overline{\Gamma} $, and
		\item[(iii)] $ (\Gamma(\xi^*)- \overline{\Gamma})\cdot \xi^*=\mathbf{0} $.
	\end{enumerate}
\end{cor}
The particularly simple structure of the Bradley--Terry paired comparison model allows a direct graphical interpretation of the design as graph and the design weights as edge weights.
Each choice set and entry of $ \xi $ corresponds to one pair of alternatives which we always linearize in lexicographic order, i.e.\ $ \xi=(w_{12},w_{13},\ldots w_{(m-1)m}) $.
Equation~\eqref{eq:edge_weights} simplifies to
\[Q_{uv}(\xi)=\lambda_{uv}w_{uv}.\]
Consequently, a vanishing design weight is a non-edge in the graph representation of the information matrix.
Corollary~\ref{cor:KKT-conditions}~(iii) thus inflicts sparsity in the graph. 

We now study the graphical representation of a $ D $-optimal design for the Bradley--Terry paired comparison model in detail.
A graph $ G=(V,E) $ is called decomposable when it is a complete graph or when its vertex set $ V $ can be written as a union $ V=V_1\cup V_2 $ where the induced subgraph with vertex set $ V_1\cap V_2 $ is a complete graph and the induced subgraphs with vertex sets $ V_1,V_2 $ are both decomposable.
A graph is decomposable if and only if it is chordal, that is all its cycles of length four or more have a chord.
The following theorem illustrates how to uniquely obtain the design weights as rational functions in $ \pi $, when the support graph of the design is decomposable. 
In accordance with the language of graphical models, we refer to the vertex sets of complete subgraphs as cliques and to the intersection of two cliques as separators.

\begin{theorem}\label{thm:decomposable}
	Let $ \xi^* $ be a $ D $-optimal design for the parameter~$\pi$.
	When the graph $ G=([m],\text{supp}(\xi^*)) $ is decomposable, then $ \xi^* $ is a rational function in~$\pi$.
	Precisely, the unique $ D $-optimal design $ \xi^* $ is recovered from $ \xi_{uv}^*=-\frac{M_{uv}(\xi^*,\pi)}{\lambda_{uv}}$,
	where we obtain $M(\xi^*,\pi)$ from~$\Gamma(\xi^*,\pi)$.
\end{theorem}

Theorem~\ref{thm:decomposable} only applies to decomposable graphs.
Simulations indicate that generically, $ D $-optimal designs correspond to decomposable graphs:
\begin{conj}\label{conj:decomposable}
	In the Bradley--Terry model, a $ D $-optimal design's graph is decomposable.
\end{conj}
In the optimization problem $ \eqref{eq:dual_BT} $, the matrix $ \overline{\Gamma} $ is parameterized by the low-dimensional parameter vector $ \pi\in\RR^m $. 
This parameterization seems to enforce the chordality of the solution of $ \eqref{eq:dual_BT} $, as arbitrary sample variograms $ \overline{\Gamma} $ do not necessarily imply chordality \citep{REZ21}.
It is easy to see via Theorem~\ref{cor:KKT-conditions} that $ \eqref{eq:dual_BT} $ is equivalent to a graphical model with respect to the graph corresponding to its solution.
Results from algebraic statistics for discrete graphical models \citep{lauritzen1996,GMS2006} and Gaussian graphical models \citep{SU2010} link properties of maximum likelihood estimators with decomposable graphs.
In the spirit of these results, studying the maximum likelihood degree \citep{HS2014} of semidefinite Gaussian graphical models for sufficient statistics that depend on a lower dimensional parameterization could provide a proof for Conjecture~\ref{conj:decomposable}.
Solving the conjecture would also solve \citet[Problem 15]{KRS2021}, as the corresponding graphs are $ 4 $-cycles and therefore non-decomposable.

An extensive treatment of designs for three and four alternatives that are supported on various decomposable graphs as well as saturated designs are presented in the appendix in Section~\ref{appendix:examples}. These recover the findings of \citet{Grasshoff2008,KRS2021}, but in the improved framework of the present paper. 
We further present an example with five alternatives. 
The computed symbolic solutions for the $ D $-optimal design were unknown before, as the methods of \citet{KRS2021} were not able to solve this problem in reasonable time. 
With our new methodology, the problem becomes computationally very simple, as the problem is linear in $ \Gamma(\xi,\pi) $. The details are available in a \texttt{Mathematica} notebook.

\section{Algorithms}\label{sec:alg}

As explained in Section~\ref{sec:BradleyTerry}, Corollary~\ref{cor:dualBT} translates the optimal design problem to a Gaussian maximum likelihood estimation problem for which algorithms are available, for example those in \citet{REZ21,YCP2021}.
One such algorithm is the the block descent \citet[Algorithm~1]{REZ21}, which is implemented in the function \texttt{emtp2} in the \texttt{graphicalExtremes} \citep{graphicalExtremes} package in R.
As $ \overline{\Gamma} $ is a dually feasible point, and every optimization step in the algorithm preserves dual feasibility, convergence is guaranteed up to the numerical precision of the employed quadratic programming solver, compare \citet[p.~22]{REZ21}.
Other algorithms for this problem rely on gradient descent methods, see for example \citet{YCP2021} or \citet{EPO2017}.
For the Bradley--Terry paired comparison model it seems best to use these readily available algorithms.

For the general discrete choice problem, the linear constraint in the dual problem in Proposition~\ref{prop:dual} does not allow a simple reformulation as a coordinate-wise constraint on $ \Gamma(\xi,\pi) $ like in the Bradley--Terry paired comparison model.
This follows from the non-quadratic form of the edge-hyperedge incidence matrix $ S $.
As a consequence, the above algorithms in the previous paragraph are not applicable to solve the dual problem.
Our Algorithm~\ref{alg:DCD} below applies to a general discrete choice problem.
We employ the gradient descent algorithm \texttt{SLSQP} available in the \texttt{nloptr} package in \texttt{R} to find the unique optimal point $\Gamma(\xi^*,\pi)$.
Again using the \texttt{SLSQP} algorithm, we compute one $ D $-optimal design $ \xi^*_1 $ from the optimal $ \Gamma(\xi^*,\pi) $ by solving a quadratic program that minimizes $ ||\vec{\Q}(\xi^*,\pi)-LS^{T}R^{-1}\xi||_2 $ under the design constraints, where $ \vec{\Q}(\xi^*,\pi) $ is obtained from the optimal $ \Gamma(\xi^*,\pi) $.

\begin{algorithm} 
	\label{alg:DCD}
	$ $ \\[1ex]
	\textbf{Input:} {A parameter vector $ \pi\in\RR_{> 0}^{m} $ and the choice set size $ k $.}
	
	\noindent
	\textbf{Initialize:} Define the objective function $ \log \det\left(\begin{smallmatrix}
		0&-\mathbf{1}^T\\
		\mathbf{1}&-\frac{\Gamma(\xi)}{2}\\
	\end{smallmatrix}\right) $ and the constraint vector $ R^{-1}S L \vec{\Gamma}(\xi) - (m-1) \mathbf{1} $ and their derivatives.
	
	\noindent
	\textbf{Computation:} \begin{enumerate}
		\item Solve the optimization problem in $ \Gamma $ using the \texttt{SLSQP} algorithm.
		\item Find a $D$-optimal design from the optimal $\Gamma^*$, again using the \texttt{SLSQP} algorithm. 
	\end{enumerate}
	
	\noindent
	\textbf{Output: }{A $ D $-optimal design $ \xi_1 $, the solution $ \Gamma^* $, the directional derivatives for $ \xi_1 $, and the error in step~2.}
\end{algorithm}

As a potential measure of convergence, the duality gap is the difference of the primal and the dual objective functions, where we rewrite the primal problem to include the equality constraints:
\begin{align*}
	-\log \sum_{T\in\mathcal{T}}\prod_{uv\in T}Q_{uv}(\xi)+\langle (m-1) \mathbf{1},\xi-\frac{1}{\binom{m}{k}}\mathbf{1}\rangle -\log \det\begin{pmatrix}
		0&-\mathbf{1}^T\\
		\mathbf{1}&-\frac{\Gamma(\xi)}{2}\\
	\end{pmatrix}
	=(m-1) (\mathbf{1}^T \xi -1).
\end{align*}
In our computations, this function is used to assess convergence of the optimization procedure.
A gap of less than $10^{-16}$ is considered as zero and the optimization problem as solved.

\subsection{Performance}
We study the performance of our implementation of Algorithm~\ref{alg:DCD} for $ m\in\{8,10\} $ and $ k\in\{3,4,5,6\} $. The parameter $ \pi $ is sampled uniformly from $[1,20]^d$.
Table~\ref{tab:performance} shows the largest value of the directional derivatives for the computed design, the duality gap and the computation time, averaged over $ n=10 $ simulations.
Note that these examples are already quite high-dimensional with respect to the design. For example for $ m=10 $ and $ k=5 $, there are $ \binom{10}{5}=252 $ different choice sets, while $ \vec{\Gamma}(\xi,\pi) $ only has $ \binom{10}{2}=45 $ entries.
As a result for a growing number of choice sets, obtaining the optimal (and unique) $ \Gamma(\xi^*,\pi) $ is much less expensive than the quadratic program that computes a $ D $-optimal design from the optimal $ \Gamma(\xi^*,\pi) $.
The computation was conducted on a standard laptop.

\begin{table}[ht]
	\centering
	\addtolength{\tabcolsep}{-1.2pt}    
	\begin{tabular}{rrrrrrrrr}
		\hline
		&$  m=8 $ & $ m=8 $ & $ m=8 $ & $ m=8 $ & $ m=10 $ & $ m=10 $ & $ m=10 $ & $ m=10 $ \\ 
		&$  k=3 $ & $ k=4 $ & $ k=5 $ & $ k=6 $ & $ k=3 $ & $ k=4 $ & $ k=5 $ & $ k=6 $ \\ 
		\hline
		Dir.~der.& 5.12e-8 & 3.16e-8 & 6.02e-8 & 2.21e-8 & 6.78e-8 & 3.74e-4 & 5.30e-5 & 1.68e-7 \\ 
		Dual.~gap & 0 & 0 & 0 & 0 & 0 & 2.00e-16 & 0 & 0 \\ 
		Step (1) time & 0.05 & 0.05 & 0.05 & 0.05 & 0.18 & 0.24 & 0.27 & 0.24 \\ 
		Step (2) time & 0.16 & 0.43 & 0.19 & 0.02 & 2.40 & 29.40 & 49.31 & 21.54 \\ 
		\hline
	\end{tabular}
	\addtolength{\tabcolsep}{1.2pt}    
	\caption{The performance table shows the averaged directional derivatives, duality gaps and computation times in seconds for the two steps in Algorithm~\ref{alg:DCD}. We observe that finding the $ D $-optimal $ \Gamma $ in step (1) is fast, but deriving a $ D $-optimal design from the $ D $-optimal $ \Gamma $ with high precision is more expensive with growing dimension.}\label{tab:performance}
\end{table}

\section{Applications}\label{sec:applications}

In this section, we demonstrate our new methodology in applications.
The $ D $-efficiency of a design $ \xi $ for a parameter $ \pi $ is defined as
\[\text{eff}_{D}(\xi,\pi)=\left(\frac{\det(M^{(m)}(\xi,\pi_0))}{\det(M^{(m)}(\xi^*,\pi_0))}\right)^{\frac{1}{m-1}},\]
where $ \xi^* $ is a locally $ D $-optimal design for $ \pi $.
This means that a $ D $-optimal design has $ D $-efficiency one.
The $ D $-efficiency of a design describes the loss of information caused by a non-optimal design.
For example, an efficiency of $ \frac{1}{2} $ implies that twice the amount of observations is needed to obtain the same information as when using an optimal design.
The computation of a $ D $-optimal matrix $ \Gamma(\xi^*,\pi) $ allows us to evaluate the $ D $-efficiencies of specific, common discrete choice designs.
As a first application, we study a poll dataset that investigates the perception of climate change.

\subsection{Perception of climate change}

We study the discrete choice dataset \texttt{icons} available in the \texttt{R} package \texttt{hyper2} \citep{Hankin2017}.
In the study, 133 participants from Norfolk, UK were asked to select among $ k=4 $ out of $ m=6 $ climate change concerns the icon that they perceive as most concerning.
The icons are NB (flooding of the Norfolk Broads national park), L (London flooding due to sea level rise), PB (Polar Bear extinction), THC (slowing or stop of the thermo-haline circulation), OA (oceanic acidification) and WAIS (West Antarctic Ice Sheet melting).
The responses are presented in Table~\ref{tab:icons}. We observe that out of $ \binom{6}{4} $ possible choice sets, the study design assigns varying proportions of observations to $ 9 $ different choice sets, which thus constitute the support of the study design.
\begin{table}
	\centering
	\begin{tabular}{ccccccc}
		\hline
		choice set& NB & L & PB & THC & OA & WAIS \\ 
		\hline
		1 &   5 &   3 &  &   4 &  &   3 \\ 
		2 &   3 &  &   5 &   8 &  &   2 \\ 
		3 &  &   4 &   9 &   2 &  &   1 \\ 
		4 &  10 &   3 &  &   3 &   4 &  \\ 
		5 &   4 &  &   5 &   6 &   3 &  \\ 
		6 &  &   4 &   3 &   1 &   3 &  \\ 
		7 &   5 &   1 &  &  &   1 &   2 \\ 
		8 &   5 &  &   1 &  &   1 &   1 \\ 
		9 &  &   9 &   7 &  &   2 &   0 \\ 
		\hline
	\end{tabular}
	\caption{Responses of 133 participants presented with choice sets of size 4 from a set of icons NB, L, PB, THC, OA and WAIS (see text for explanation). Each row corresponds to one choice set, such that the entries in the table correspond to the selection of each icon in the respective choice set.}\label{tab:icons}
\end{table}
The \texttt{hyper2} package provides a maximum likelihood estimate 
\begin{align*}
	\hat{\pi}&=( \pi_\text{NB},\hat{\pi}_\text{L},\hat{\pi}_\text{PB},\hat{\pi}_\text{THC},\hat{\pi}_\text{OA},\hat{\pi}_\text{WAIS})=(0.2523, 0.1736, 0.2246, 0.1701, 0.1107, 0.0687).
\end{align*} 
We compute the $ D $-optimal approximate design for $ \hat{\pi} $ with Algorithm~\ref{alg:DCD}.
The $ D $-efficiencies for the study design, the complete design with constant design weights, which in this setting is the only possible balanced incomplete block design, and a rounded version (i.e.~rounded such that $ w_i=\frac{n_i}{133} $ for $ n_i\in \NN_0 $) of the $ D $-optimal design are as follows, showing that the study design has the lowest efficiency.
\begin{center}
	\begin{tabular}{rrrr}
		\hline
		& study & complete & rounded \\ 
		\hline
		$ D $-efficiency & 0.95173 & 0.96643 & 0.99997 \\ 
		\hline
	\end{tabular}
\end{center}

\subsection{Cricket}
We study the \texttt{T20} dataset from \texttt{hyper2} \citep{Hankin2017}.
It contains 633 cricket matches results between 13 teams in the Indian Premier League for the period from 2008 to 2017, with seven draws and three no-result matches removed.
The package provides a maximum likelihood estimate $ \hat{\pi} $ for the playing strength of each team in a Bradley--Terry paired comparison model:
\[ \scriptstyle
\hat{\pi}=(0.1177,0.0503,0.0614,0.0634,0.0867,0.0571,0.0724,0.1106,0.0296,0.0767,0.0816,0.0926,0.0999)
\]
The team names are available in the \texttt{hyper2} package.
We compute the $ D $-optimal design in $ \hat{\pi} $, the efficiencies of the observed design and of the complete design.
The results are as follows.
\vspace{1ex}
\begin{center}
	\begin{tabular}{rrr}
		\hline
		& observed & complete \\ 
		\hline
		$ D $-efficiency & 0.77507 & 0.98792 \\ 
		\hline
	\end{tabular}
\end{center}
\vspace{1ex}
Thus the observed design has comparably low efficiency.
Although the match scheduling was obviously not chosen to optimize the $ D $-efficiency for learning the playing strength, this example demonstrates the sensitivity of the $ D $-efficiency.
Our methodology allows us to quantify the loss of $ D $-efficiency and to certify that the constant weight complete design is more suitable. 

\subsection{Simulated parameters}\label{sec:simulation}
We assume that $ m=6 $ and $ k=3 $ and simulate parameters $ \beta_u,\;1\le u\le 6 $ from independent centered normal distributions with standard deviation~$\sigma$. This implies that $ \pi_u,\;1\le u\le 6 $ are sampled from a log-normal distribution.
For the initial guess $ \beta=\bs 0 $, that is $ \pi=\bs 1 $, the choice probabilities are all equal. 
In this case, the complete design with equal weights is $ D $-optimal.
For each of the standard deviations $ \sigma\in\{0.5,1,1.5,2\} $ we sampled $n=1000$ parameter vectors and computed the $ D $-optimal design with our implementation of Algorithm~\ref{alg:DCD}, as well as the $D$-efficiency of the complete design with equal weights for that parameter.
The results are shown in Table~\ref{tab:simulation}.
As the standard deviation increases, the sampled parameters lie further from the origin, which should correspond to decreasing efficiency for the complete design with equal weights (compare also Section~\ref{sec:eff_line}).
Indeed, we observe this behavior in the simulation.
Furthermore, the $ D $-optimal designs have decreasing support size, which aligns well with similar observations for the Bradley--Terry paired comparison model that were discussed by \citet{KRS2021}.
\begin{table}[ht]
	\centering
	\begin{tabular}{rrrrr}
		\hline
		$ \sigma $	& 0.5 & 1 & 1.5 & 2 \\ 
		\hline
		Directional derivatives & 5.88e-12 & 9.97e-09 & 3.34e-08 & 1.97e-05 \\ 
		Duality gap & 0 & 1.11e-16 & 1.11e-16 & 1.08e-16 \\ 
		$ D $-efficiency & 0.9943 & 0.8930 & 0.6832 & 0.5414 \\ 
		Mean support size & 19.9840 & 11.2850 & 6.8740 & 5.9170 \\ 
		Time & 0.0153 & 0.0296 & 0.0382 & 0.0473\\ 
		\hline
	\end{tabular}
	\caption{Results for logarithmic parameters sampled from a normal distribution. The table shows the average of the largest directional derivative, the duality gap, the efficiency of the complete design with equal weights, the support size of the $ D $-optimal design, and the computation time of Algorithm~\ref{alg:DCD}. We observe decreasing efficiency of the complete design with equal design weights for increasing standard deviation of the parameter.}\label{tab:simulation}
\end{table}

\subsection{Efficiency comparison}\label{sec:eff_line}
Here we investigate the efficiency of the complete design with constant design weights and two balanced incomplete block designs (BIBDs) for $ m=6 $ and $ k=3 $.
The first BIBD is
\[\xi_1=\frac{1}{10}(0,1,0,1,0,1,1,1,0,0,1,1,0,0,0,1,0,1,0,1)^T,\]
where the individual choice sets can be recovered from the columns of the matrix $ S^T $ in Example~\ref{ex:smallest}.
A second BIBD is $ \xi_2=\frac{1}{10}\mathbf{1}-\xi_1 $. 
Clearly, $ \xi_1 $ and $ \xi_2 $ have complementary support.
These designs are also available in the \texttt{R} implementation.
To evaluate the $D$-efficiency in a given distance from the true parameter, we compute it on a line in logarithmic parameter space starting in the origin.
We evaluate 100 points on the line parameterized by 
\begin{align*}
	\pi=(\pi_1,\pi_1^{1/2},\pi_1^{5/4},\pi_1^{7/4},\pi_1^{3/4},1),
\end{align*}
where $ \pi_1=\exp(\ell/10) $ and $ \ell=0,1,\ldots,99 $.
The resulting efficiencies are plotted in Figure~\ref{fig:efficiency}.
We observe that all designs are $ D$-optimal for $ \pi=\bs 1 $, as explained by design theory.
Again, increasing Euclidean distance to the origin decreases the $ D $-efficiency of the complete design with constant design weights.
For $ \ell=100 $, we compute the $ D $-optimal design
\[\xi^{*}\approx(0,0,0.19196,0,0.20127,0.21450,0,0,0,0,0,0,0,0,0,0.39228,0,0,0,0).\]
The support of $ \xi^* $ is contained in the support of $ \xi_1 $.
Besides, the $ D $-efficiency of the complete design with constant design weights lies between the $ D $-efficiencies of $ \xi_1 $ and $ \xi_2 $ in Figure~\ref{fig:efficiency}, further suggesting that this results from the complementary supports of the designs.
\begin{figure}
	\centering
	\includegraphics[scale=0.5]{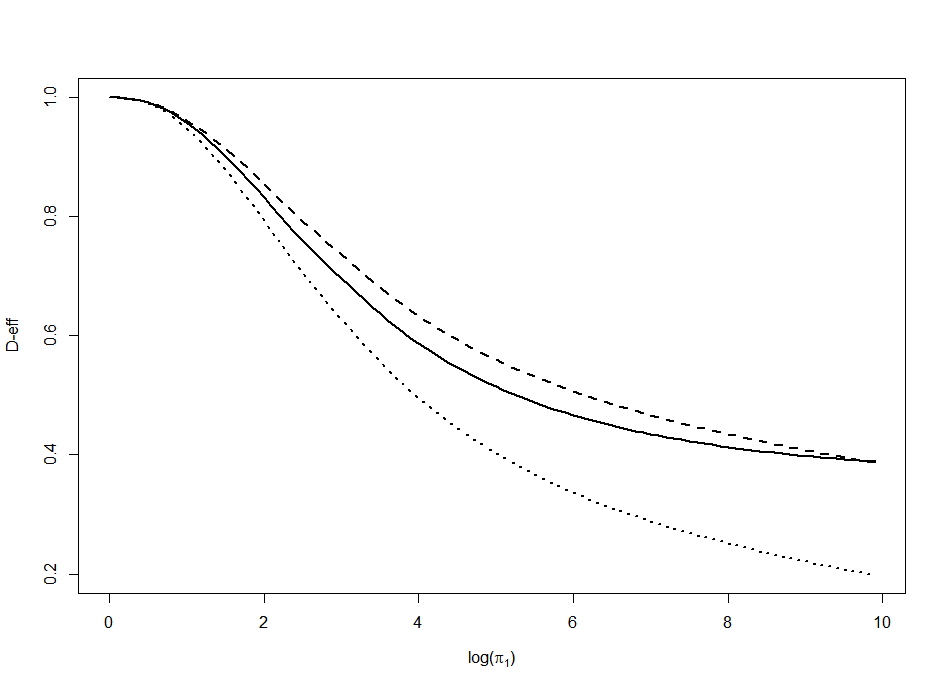}
	\caption{Efficiencies for 100 points on a line starting in the origin for three designs. The solid line corresponds to the complete design with constant weights, the dashed to the BIBD $ \xi_1 $ and the dotted line to the BIBD $ \xi_2 $.}\label{fig:efficiency}
\end{figure}

\subsection*{Acknowledgements}
Frank R\"ottger was supported by the Swiss National Science Foundation (Grant 186858).
Thomas Kahle was partially supported by the Deutsche Forschungsgemeinschaft (DFG,
German Research Foundation) -- 314838170, GRK~2297 MathCoRe.

\bibliographystyle{chicago}
\bibliography{bibliography2}

\appendix

\section{Proof of Proposition~\ref{prop:dual}}
\begin{proof}
	Let 
	\[f(\xi)=\begin{cases}
		-\log \sum_{T\in\mathcal{T}}\prod_{uv\in T}Q_{uv}(\xi),& \text{when } \xi \in\Delta_{\binom{m}{k}},\\
		+\infty, &\text{when } \xi \in \RR^{\binom{m}{k}}\setminus \Delta_{\binom{m}{k}}, \\
	\end{cases}\]
	be an extended real-valued function, such that $ f(\xi) $ equals the objective function of the $ D $-criterion when $ \xi $ is a design and $ +\infty $ otherwise.
	Now, the $ D $-criterion can be reformulated as 
	\begin{align}
		\operatorname{minimize}\quad f(\xi).\label{eq:opt_problem_DC_equiv}
	\end{align}
	The Lagrangian of the $ D $-criterion is
	\begin{align*}
		\mathcal{L}(\xi,\B,\mu)=-\log \sum_{T\in\mathcal{T}}\prod_{uv\in T}Q_{uv}(\xi)-\langle \B,\xi \rangle +\langle \mu \mathbf{1},\xi-\frac{1}{\binom{m}{k}}\mathbf{1}\rangle
	\end{align*}
	with $ \B\in \RR^{\binom{m}{k}}_{\ge 0} $, $ \mu \in \RR $.
	We observe that
	\[\sup_{\B\ge 0,\mu}\mathcal{L}(\xi,\B,\mu)=\begin{cases}
		-\log \sum_{T\in\mathcal{T}}\prod_{uv\in T}Q_{uv}(\xi), & \text{when } \xi \in\Delta_{\binom{m}{k}},\\
		+\infty, & \text{when } \xi \in \RR^{\binom{m}{k}}\setminus \Delta_{\binom{m}{k}}.\\
	\end{cases}\]
	It follows that \eqref{eq:opt_problem_DC_equiv} can be rewritten as
	\begin{align}
		\inf_{\xi}\sup_{B\ge 0,\mu}\mathcal{L}(\xi,\B,\mu).\label{eq:infsup}
	\end{align}
	As all constraints in \eqref{eq:infsup} are linear, duality theory (Slater's condition) infers that \eqref{eq:infsup} is invariant under swapping the infimum and supremum. 
	The Lagrange dual function is
	\begin{align*}
		\inf_{\xi}\mathcal{L}(\xi,\B,\mu).
	\end{align*}
	According to \citet[Proposition~A.5]{REZ21}, it holds that $ \nabla_{Q}\log \sum_{T\in\mathcal{T}}\prod_{uv\in T}Q_{uv}=\Gamma$.
	It then follows from the multivariable chain rule that 
	\begin{align*}
		\frac{\partial}{\partial _{w_j}}\left(\log \sum_{T\in\mathcal{T}}\prod_{uv\in T}Q_{uv}(\xi)\right)&=\frac{1}{(\sum_{i\in C_j}\pi_i )^2}\sum_{uv\in C_j}\pi_{u}\pi_{v}\Gamma_{uv}(\xi)\\
		\intertext{such that we obtain}
		\nabla_{\xi}\left(\log \sum_{T\in\mathcal{T}}\prod_{uv\in T}Q_{uv}(\xi)\right)&=R^{-1}S L \vec{\Gamma}(\xi).
	\end{align*}
	The Karush-Kuhn-Tucker conditions are
	\begin{align}
		-R^{-1}S L \vec{\Gamma}(\xi)  -\B+\mu \mathbf{1}&=0,\label{eq:KKT1}\\
		\xi&\ge 0,\label{eq:KKT2}\\
		\langle \xi, \mathbf{1}\rangle&=1,\label{eq:KKT3}\\
		\langle \B, \xi\rangle &=0,\label{eq:KKT4}\\
		\B&\ge0.\label{eq:KKT5}
	\end{align}
	The first condition \eqref{eq:KKT1}, $B =-R^{-1}S L \vec{\Gamma}(\xi)  +\mu \mathbf{1}$ together with the dual feasibility $ \B\ge 0 $, yields $ R^{-1}S L \vec{\Gamma}(\xi)\le \mu \mathbf{1} $.
	We therefore refer to all $\Gamma(\xi)\in \mathcal{C}^d$ satisfying $R^{-1}S L \vec{\Gamma} \le\mu \mathbf{1}$ as dually feasible points.
	As a consequence, evaluating the Lagrangian in the optimal point gives
	\begin{align*}
		&-\log \sum_{T\in\mathcal{T}}\prod_{uv\in T}Q_{uv}(\xi)- \langle -R^{-1}S L \vec{\Gamma}(\xi)  +\mu \mathbf{1},\xi \rangle +\langle \mu \mathbf{1},\xi-\frac{1}{\binom{m}{k}}\mathbf{1}\rangle\\
		&=-\log \sum_{T\in\mathcal{T}}\prod_{uv\in T}Q_{uv}(\xi)+\langle \vec{\Gamma}(\xi),\vec{\Q}(\xi)\rangle-\mu\\
		&=-\log \sum_{T\in\mathcal{T}}\prod_{uv\in T}Q_{uv}(\xi)+(m-1)-\mu.
	\end{align*}
	We further obtain from the fourth condition \eqref{eq:KKT4} that in the case of optimality
	\begin{align*}
		\langle \B,\xi \rangle&= (m-1)-\mu \langle \xi, \mathbf{1}\rangle=(m-1)-\mu=0,
	\end{align*}
	such that $ \mu=m-1 $. This implies that the dual objective function is
	\[	-\log \sum_{T\in\mathcal{T}}\prod_{uv\in T}Q_{uv}(\xi)=-\log\det(M^{(m)}(\xi))=\log\det(\Sigma^{(m)}(\xi)). \]
	The Cayley--Menger determinant
	\[\det(\Sigma^{(m)}(\xi))=\det\begin{pmatrix}
		0&-\mathbf{1}^T\\
		\mathbf{1}&-\frac{\Gamma(\xi)}{2}\\
	\end{pmatrix}\]
	gives the proposition.
\end{proof}

\section{Connections of Theorem~\ref{thm:KKT-conditions} with the Kiefer--Wolfowitz theorem}
Following \citet{silvey1980optimal}, the directional derivatives of $ \log\det(M^{(m)}(\xi,\pi_0)) $ towards $ \log\det(M^{(m)}(C_j,\pi_0)) $ are given for each $ 1\le j\le \binom{m}{k} $ by
\begin{align}
	\langle M^{(m)}(C_j,\pi_0),M^{(m)}(\xi,\pi_0)^{-1}\rangle - (m-1).\label{eq:dir_der}
\end{align}
Via the Farris transform, \eqref{eq:dir_der} equals
\begin{align}
	\dla Q(C_j,\pi),\Gamma(\xi,\pi) \dra-(m-1),\label{eq:dir_der_Gamma-M-space}
\end{align}
where $ Q(C_j,\pi) $ are the edge weights of the graph with Laplacian matrix $ M(C_j,\pi) $.
The directional derivatives connect to $ D $-optimality for discrete choice designs via the following application of the celebrated Kiefer--Wolfowitz equivalence theorem (see for example \citet{silvey1980optimal}):	
\begin{theorem}\label{t:silvey}
	A discrete choice design $ \xi^* $ is locally $ D $-optimal in $ \pi $ if and only if 
	\begin{align}\label{eq:optimalityregions}
		\dla Q(C_j,\pi),\Gamma(\xi^*,\pi) \dra&\le m-1
	\end{align}
	for all $ 1\le j\le \binom{m}{k} $.
\end{theorem}

According to Lemma~\ref{lem:Q_in_xi}, we have $ \vec{Q}(C_j,\pi)=LS^{T}R^{-1}e_j $, i.e.~$ \vec{Q}(C_j,\pi) $ equals the $ j $-th column of $ LS^{T}R^{-1} $.
This means that the left hand sides of \eqref{eq:optimalityregions} are jointly expressed by
\[(LS^{T}R^{-1})^T \vec{\Gamma}(\xi^*,\pi)=R^{-1}S L \vec{\Gamma}(\xi^*,\pi), \]
which shows the equivalence of \eqref{eq:optimalityregions} and Theorem~\ref{thm:KKT-conditions}(ii).
According to {\citet[Corollary~3.10]{silvey1980optimal}},
for choice sets $C_j$ with positive weight in a locally $ D $-optimal design $ \xi^{*} $, the inequalities \eqref{eq:optimalityregions} in Theorem~\ref{t:silvey} hold with equality.
This is equivalent to Theorem~\ref{thm:KKT-conditions}(iii), while  Theorem~\ref{thm:KKT-conditions}(i) ensures that $ \xi^* $ is a design.

\section{Example for the application of Theorem~\ref{thm:KKT-conditions}}\label{appendix:example}

\begin{example}\label{ex:smallest}
	Let $ m=6 $ and $ k=3 $. The edge-hyperedge incidence matrix equals
	\[\arraycolsep=5pt\def\arraystretch{0.5} S^T=\begin{pNiceArray}[
		first-row,code-for-first-row=\scriptstyle,
		last-col,code-for-last-col=\scriptstyle,
		]{cccccccccccccccccccc}
		1&2&3&4&5&6&7&8&9&10&11&12&13&14&15&16&17&18&19&20&\\
		1&1&1&1&0&0&0&0&0&0&0&0&0&0&0&0&0&0&0&0&(1,2)\\
		1&0&0&0&1&1&1&0&0&0&0&0&0&0&0&0&0&0&0&0&(1,3)\\
		0&1&0&0&1&0&0&1&1&0&0&0&0&0&0&0&0&0&0&0&(1,4)\\
		0&0&1&0&0&1&0&1&0&1&0&0&0&0&0&0&0&0&0&0&(1,5)\\
		0&0&0&1&0&0&1&0&1&1&0&0&0&0&0&0&0&0&0&0&(1,6)\\
		1&0&0&0&0&0&0&0&0&0&1&1&1&0&0&0&0&0&0&0&(2,3)\\
		0&1&0&0&0&0&0&0&0&0&1&0&0&1&1&0&0&0&0&0&(2,4)\\
		0&0&1&0&0&0&0&0&0&0&0&1&0&1&0&1&0&0&0&0&(2,5)\\
		0&0&0&1&0&0&0&0&0&0&0&0&1&0&1&1&0&0&0&0&(2,6)\\
		0&0&0&0&1&0&0&0&0&0&1&0&0&0&0&0&1&1&0&0&(3,4)\\
		0&0&0&0&0&1&0&0&0&0&0&1&0&0&0&0&1&0&1&0&(3,5)\\
		0&0&0&0&0&0&1&0&0&0&0&0&1&0&0&0&0&1&1&0&(3,6)\\
		0&0&0&0&0&0&0&1&0&0&0&0&0&1&0&0&1&0&0&1&(4,5)\\
		0&0&0&0&0&0&0&0&1&0&0&0&0&0&1&0&0&1&0&1&(4,6)\\
		0&0&0&0&0&0&0&0&0&1&0&0&0&0&0&1&0&0&1&1&(5,6)\\
	\end{pNiceArray}\;.\]
	Here the choice sets $C_{1},\dots, C_{\binom{6}{3}}$ are ordered lexicographically, starting with $\{1,2,3\}$ and going to $C_{20} = \{4,5,6\}$.  The entries are simple indicators of the inclusion of an edge in a choice set.  For example, the choice set $ C_1=\{1,2,3\} $ allows for the edges $ (1,2),(1,3),(2,3) $, which is encoded in the first column of $ S^T $.
	For an arbitrary choice set $ C_j=\{s,t,u\} $, we obtain the following inequality from Theorem~\ref{thm:KKT-conditions}~(ii) to describe $D$-optimality:
	\[
	\frac{\pi_s\pi_t}{(\pi_s+\pi_t+\pi_u)^2}\Gamma_{st}(\xi^*)+\frac{\pi_s\pi_u}{(\pi_s+\pi_t+\pi_u)^2}\Gamma_{su}(\xi^*)+\frac{\pi_t\pi_u}{(\pi_s+\pi_t+\pi_u)^2}\Gamma_{tu}(\xi^*)\le (m-1).
	\]
	By Theorem~\ref{thm:KKT-conditions}~(iii), equality holds in the inequality if and only if $w_j^*=0$.
	Note that the inequalities and equations are linear in $ \Gamma(\xi^*) $.
\end{example}

\section{Proof of Theorem~\ref{thm:decomposable}}
\begin{proof}
	When $ w_{uv}^*>0 $, it is $ \Gamma_{uv}(\xi^*,\pi)=\frac{m-1}{\lambda_{uv}} $
	according to Corollary~\ref{cor:KKT-conditions} in the case of $ D $-optimality.
	After reordering the alternatives according to the cliques $ K_1,K_2,\ldots $ with separator sets $ D_{12},D_{13},\ldots $ we have that each submatrix $ \Gamma_{K_u\cup K_v}(\xi^*,\pi) $ is of the form
	\begin{align*}
		\Gamma_{K_u\cup K_v}(\xi^*,\pi)&=\begin{pmatrix}
			\Gamma_{K_u\setminus K_v}&\Gamma_{K_u\setminus K_v,D_{uv}}&*\\
			\Gamma_{D_{uv},K_u\setminus K_v}&\Gamma_{D_{uv}}&\Gamma_{D_{uv},K_v\setminus K_u}\\
			*&\Gamma_{K_v\setminus K_u,D_{uv}}&\Gamma_{K_v\setminus K_u}\\
		\end{pmatrix},
	\end{align*} 
	with $ * $ denoting the unknown entries. \citet{Hentschel2021}, showed via positive definite matrix completion \citep[Theorem 5.3]{KNS97} that $ \Gamma(\xi^*,\pi) $ is uniquely recovered as the conditionally negative definite completion such that $ M(\xi^*,\pi) $ is the Laplacian of the underlying decomposable graph with edge weights $ w^*_{uv}\lambda_{uv} $.
	As a consequence $ \Gamma(\xi^*,\pi) $ is a rational function in $ \pi $.
	The design weights are then obtained from $ \Gamma(\xi^*,\pi) $, such that $ w_{uv}^*=-\frac{M(\xi^*,\pi)_{uv}}{\lambda_{uv}} $.
\end{proof}

\vspace*{-10pt}

\section{Solutions for the Bradley--Terry paired comparison model}\label{appendix:examples}
In this section we discuss decomposable examples for the Bradley--Terry paired comparison model. We begin with a general discussion of saturated designs for an arbitrary number of alternatives, compare \citet{KRS2021}.
For three and four alternatives, \citet{Grasshoff2008,KRS2021} found explicit formulas, expressing the $D$-optimal design weights in terms of the parameters, thereby solving the optimal design problem.
We first express these solutions much more efficiently in our new framework.
Then in Section~\ref{sec:five_alternatives}, we give previously unknown solutions for five alternatives.
\subsection{Saturated Designs}

A saturated design is a design with minimal support with nonsingular reduced information matrix. For the Bradley--Terry paired comparison model with $ m $ alternatives, saturated designs are supported on $ m-1 $ design points such that the design is supported on a tree.
The information matrix is the Laplacian matrix of a tree with edge weights $ \lambda_{uv}w_{uv} $. 
Designs supported on trees have a particularly simple structure in the completion of $ \Gamma(\xi^*,\pi) $, as it corresponds to a tree metric. This means that the missing entries compute for all $ 1\le u < v \le m $ as
\begin{align*}
	\Gamma_{uv}(\xi^*,\pi)=\sum_{st\in\text{ph}(u,v)}\Gamma_{st}(\xi^*,\pi),
\end{align*}
where $ \text{ph}(u,v) $ is the unique path between $ u $ and $ v $.
\citet{KRS2021} showed that $ D $-optimal saturated designs always correspond to paths, i.e.\ trees where all nodes have degree at most two.
\citet{KRS2021} further shows that all saturated designs can be recovered via permutations of the standard path  $ T=1-2-3-\ldots-m $.
For the design supported on the standard path, it follows that $ \Gamma_{uv}(\xi^*,\pi)=\frac{m-1}{\lambda_{uv}} $ for $ uv\in E(T) :=\{12,23,\ldots(m-1)m\} $.
It follows from \citet[Cor.~2.5]{KNS97}, that on trees it holds that
\begin{align*}
	\Sigma_{uv}(\xi^*,\pi)&=-\sum_{uv \in \text{ph}(m,\text{lca}(u,v))}\frac{1}{M_{uv}(\xi^*,\pi)},\\
	\Sigma_{uu}(\xi^*,\pi)&=-\sum_{uv \in \text{ph}(m,u)}\frac{1}{M_{uv}(\xi^*,\pi)}.
\end{align*}
Here, $ \text{ph}(m,\text{lca}(u,v)) $ denotes the path from $ m $ to the last common ancestor of $ u $ and $ v $ or, if $ u $ is a descendant of $ v $, the path from $ m $ to $ v $.

Now, as $ \Gamma_{uv}(\xi^*,\pi)= \Sigma_{uu}(\xi^*,\pi)+\Sigma_{vv}(\xi^*,\pi)-2\Sigma_{uv}(\xi^*,\pi) $, it holds that for $ uv\in E(T) $ we have $ \Gamma_{uv}(\xi^*,\pi)=-\frac{1}{M_{uv}(\xi^*,\pi)} $.
Therefore for $ uv\in E(T) $, we recover the well known fact that for saturated designs the nonzero design weights are all equal:
\[w_{uv}=-\frac{M_{uv}(\xi^*,\pi)}{\lambda_{uv}}=\frac{1}{\Gamma_{uv}(\xi^*,\pi) \lambda_{uv}}=\frac{1}{m-1}.\]
By Corollary~\ref{cor:KKT-conditions}, this design is $ D $-optimal if and only if $ \Gamma_{uv}(\xi^*,\pi)\le\frac{m-1}{\lambda_{uv}} $ for all $ uv\not\in E(T) $.

\subsection{Three alternatives}
\begin{example}\label{ex:m=3}
	Let $ m=3 $. Then, the $ 3\times 3 $ information matrix of a design $ \xi $ equals
	\[M(\xi,\pi)=\begin{pmatrix}
		\lambda_{12} w_{12}+\lambda_{13} w_{13} & -\lambda_{12} w_{12} &-\lambda_{13} w_{13} \\
		-\lambda_{12} w_{12} & \lambda_{12} w_{12}+\lambda_{23} w_{23} & -\lambda_{23} w_{23} \\
		-\lambda_{13} w_{13}&-\lambda_{23} w_{23}&\lambda_{13} w_{13}+\lambda_{23} w_{23}\\
	\end{pmatrix}.
	\]
	By Corollary~\ref{cor:KKT-conditions}, a design $ \xi^*=(w_{12}^*,w_{13}^*,w_{23}^*) $ with $ w_{uv}^*>0 $ for all $ 1\le u<v\le 3 $ is $ D $-optimal when
	\[\Gamma(\xi^*,\pi)=\begin{pmatrix}
		0&\frac{2}{\lambda_{12}}&\frac{2}{\lambda_{13}}\\
		\frac{2}{\lambda_{12}}&0&\frac{2}{\lambda_{23}}\\
		\frac{2}{\lambda_{13}}&\frac{2}{\lambda_{23}}&0\\
	\end{pmatrix}.\]
	Via the inverse Farris transform \eqref{eq:invFarris}, we compute
	\begin{align*}
		&\Sigma^{(3)}(\xi^*,\pi_0)=
		\begin{pmatrix}
			\frac{2}{\lambda_{13}} & -\frac{1}{\lambda_{12}}+\frac{1}{\lambda_{13}}+\frac{1}{\lambda_{23}} \\
			-\frac{1}{\lambda_{12}}+\frac{1}{\lambda_{13}}+\frac{1}{\lambda_{23}} & \frac{2}{\lambda_{23}} \\
		\end{pmatrix}.
	\end{align*}
	Inverting $ \Sigma^{(3)}(\xi^*,\pi_0) $ gives the  $ 2\times 2 $ information matrix
	\begin{align*}
		&M^{(3)}(\xi^*,\pi_0)=\\
		&\begin{pmatrix}
			-\frac{2 \lambda_{12}^2 \lambda_{13}^2 \lambda_{23}}{\lambda_{12}^2 (\lambda_{13}-\lambda_{23})^2-2 \lambda_{12} \lambda_{13} \lambda_{23} (\lambda_{13}+\lambda_{23})+\lambda_{13}^2 \lambda_{23}^2} & \frac{\lambda_{12} \lambda_{13} \lambda_{23} (\lambda_{12} (\lambda_{13}+\lambda_{23})-\lambda_{13} \lambda_{23})}{\lambda_{12}^2 (\lambda_{13}-\lambda_{23})^2-2 \lambda_{12} \lambda_{13} \lambda_{23} (\lambda_{13}+\lambda_{23})+\lambda_{13}^2 \lambda_{23}^2} \\
			\frac{\lambda_{12} \lambda_{13} \lambda_{23} (\lambda_{12} (\lambda_{13}+\lambda_{23})-\lambda_{13} \lambda_{23})}{\lambda_{12}^2 (\lambda_{13}-\lambda_{23})^2-2 \lambda_{12} \lambda_{13} \lambda_{23} (\lambda_{13}+\lambda_{23})+\lambda_{13}^2 \lambda_{23}^2} & -\frac{2 \lambda_{12}^2 \lambda_{13} \lambda_{23}^2}{\lambda_{12}^2 (\lambda_{13}-\lambda_{23})^2-2 \lambda_{12} \lambda_{13} \lambda_{23} (\lambda_{13}+\lambda_{23})+\lambda_{13}^2 \lambda_{23}^2} \\
		\end{pmatrix}.
	\end{align*}
	We obtain from \eqref{eq:gen_information} that $ M(\xi^*,\pi) $ is given by
	\begin{align*}
		M_{12}(\xi^*,\pi)&=M^{(3)}_{12}(\xi^*,\pi_0),\\
		M_{13}(\xi^*,\pi)&=-M^{(3)}_{11}(\xi^*,\pi_0)-M^{(3)}_{12}(\xi^*,\pi_0),\\
		M_{23}(\xi^*,\pi)&=-M^{(3)}_{22}(\xi^*,\pi_0)-M^{(3)}_{12}(\xi^*,\pi_0)
	\end{align*}
	and therefore
	\begin{align*}
		w_{12}^*&=-\frac{M_{12}(\xi^*,\pi)}{\lambda_{12}}=\frac{\lambda_{13} \lambda_{23} (-\lambda_{12} \lambda_{13}-\lambda_{12} \lambda_{23}+\lambda_{13} \lambda_{23})}{\lambda_{12}^2 \lambda_{13}^2-2 \lambda_{12}^2 \lambda_{13} \lambda_{23}+\lambda_{12}^2 \lambda_{23}^2-2 \lambda_{12} \lambda_{13}^2 \lambda_{23}-2 \lambda_{12} \lambda_{13} \lambda_{23}^2+\lambda_{13}^2 \lambda_{23}^2},\\
		w_{13}^*&=-\frac{M_{13}(\xi^*,\pi)}{\lambda_{13}}=\frac{\lambda_{12} \lambda_{23} (-\lambda_{12} \lambda_{13}+\lambda_{12} \lambda_{23}-\lambda_{13} \lambda_{23})}{\lambda_{12}^2 \lambda_{13}^2-2 \lambda_{12}^2 \lambda_{13} \lambda_{23}+\lambda_{12}^2 \lambda_{23}^2-2 \lambda_{12} \lambda_{13}^2 \lambda_{23}-2 \lambda_{12} \lambda_{13} \lambda_{23}^2+\lambda_{13}^2 \lambda_{23}^2},\\
		w_{23}^*&=-\frac{M_{23}(\xi^*,\pi)}{\lambda_{23}}=\frac{\lambda_{12} \lambda_{13} (\lambda_{12} \lambda_{13}-\lambda_{12} \lambda_{23}-\lambda_{13} \lambda_{23})}{\lambda_{12}^2 \lambda_{13}^2-2 \lambda_{12}^2 \lambda_{13} \lambda_{23}+\lambda_{12}^2 \lambda_{23}^2-2 \lambda_{12} \lambda_{13}^2 \lambda_{23}-2 \lambda_{12} \lambda_{13} \lambda_{23}^2+\lambda_{13}^2 \lambda_{23}^2}.\\
	\end{align*}
	By Corollary~\ref{cor:KKT-conditions}, this design is $ D $-optimal if and only if the design weights are non-negative. This allows to find a semi-algebraic description of the region of $ D $-optimality in parameter space where a design supported on all different choice sets is $ D $-optimal.
\end{example}

\subsection{Four alternatives}
We begin with the complete graph, i.e.~designs with full support.
\begin{example}\label{ex:complete_m=4}
	Let $ m=4 $. Then, the reduced information matrix $ M^{(4)}(\xi,\pi_0) $ of a design $ \xi $ that is supported on the complete graph computes as
	\[
	\begin{pmatrix}
		\lambda_{12} w_{12}+\lambda_{13} w_{13}+\lambda_{14} w_{14} & -\lambda_{12} w_{12} & -\lambda_{13} w_{13} \\
		-\lambda_{12} w_{12} & \lambda_{12} w_{12}+\lambda_{23} w_{23}+\lambda_{24} w_{24} & -\lambda_{23} w_{23} \\
		-\lambda_{13} w_{13} & -\lambda_{23} w_{23} & \lambda_{13} w_{13}+\lambda_{23} w_{23}+\lambda_{34} w_{34} \\
	\end{pmatrix}.
	\]
	Now, from Corollary~\ref{cor:KKT-conditions} the fully supported design $ \xi^* $ is $ D $-optimal when  
	\begin{align*}
		\Gamma(\xi^*,\pi)=\begin{pNiceMatrix}
			0&\frac{3}{\lambda_{12}}&\frac{3}{\lambda_{13}}&\frac{3}{\lambda_{14}}\\ 
			\frac{3}{\lambda_{12}}&0&\frac{3}{\lambda_{23}}&\frac{3}{\lambda_{24}}\\
			\frac{3}{\lambda_{13}}&\frac{3}{\lambda_{23}}&0&\frac{3}{\lambda_{34}}\\
			\frac{3}{\lambda_{14}}&\frac{3}{\lambda_{24}}&\frac{3}{\lambda_{34}}&0\\           
		\end{pNiceMatrix},
	\end{align*}
	where $ \lambda_{uv}= \frac{\pi_u \pi_v}{(\pi_u+\pi_v)^{2}} $.
	Via the inverse Farris transform, we find
	\begin{align*}
		&\Sigma^{(4)}(\xi^*,\pi_0)=
		\begin{pmatrix}
			\frac{3}{\lambda_{14}} & \frac{3}{2} \left(-\frac{1}{\lambda_{12}}+\frac{1}{\lambda_{14}}+\frac{1}{\lambda_{24}}\right) & \frac{3}{2} \left(-\frac{1}{\lambda_{13}}+\frac{1}{\lambda_{14}}+\frac{1}{\lambda_{34}}\right) \\
			\frac{3}{2} \left(-\frac{1}{\lambda_{12}}+\frac{1}{\lambda_{14}}+\frac{1}{\lambda_{24}}\right) & \frac{3}{\lambda_{24}} & \frac{3}{2} \left(-\frac{1}{\lambda_{23}}+\frac{1}{\lambda_{24}}+\frac{1}{\lambda_{34}}\right) \\
			\frac{3}{2} \left(-\frac{1}{\lambda_{13}}+\frac{1}{\lambda_{14}}+\frac{1}{\lambda_{34}}\right) & \frac{3}{2} \left(-\frac{1}{\lambda_{23}}+\frac{1}{\lambda_{24}}+\frac{1}{\lambda_{34}}\right) & \frac{3}{\lambda_{34}} \\
		\end{pmatrix}.
	\end{align*}
	We compute the information matrix $ M(\xi^*,\pi) $ from $ \Sigma^{(4)}(\xi^*,\pi_0) $ and its entry in $ uv=12 $:
	\begin{align*}
		&M_{12}(\xi^*,\pi) = \frac{1}{A}(-\lambda_{12} \lambda_{13} \lambda_{14} \lambda_{23} \lambda_{24} (\lambda_{12} (\lambda_{13} \lambda_{14} (\lambda_{23} (\lambda_{24}-\lambda_{34})-\lambda_{24} \lambda_{34})\\
		&+\lambda_{13} \lambda_{34} (\lambda_{23} (\lambda_{34}-\lambda_{24})-\lambda_{24} \lambda_{34})-\lambda_{14} \lambda_{23} \lambda_{34} (\lambda_{24}+\lambda_{34})+\lambda_{14} \lambda_{24} \lambda_{34}^2)+2 \lambda_{13} \lambda_{14} \lambda_{23} \lambda_{24} \lambda_{34}))
	\end{align*}
	with some normalization term $ A $.
	We obtain $ w_{12}=-\frac{M_{12}(\xi^*,\pi)}{\lambda_{12}} $. The design is $ D $-optimal when all weights are positive, i.e.~ when $ \xi^*\in\Delta_{6} $.
\end{example}

\begin{example}
	Let $ m=4 $ and assume a design $ \xi $ supported on the chordal graph in Figure~\ref{fig:chordal_example2}. 
	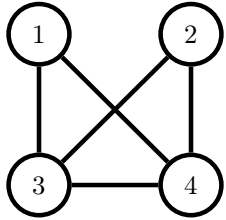
\begin{figure}[ht]\centering
		\begin{tikzpicture}[line width=.6mm]
			\node[minimum size=8mm,shape=circle,draw=black] (1) at (-1,1) {1};
			\node[minimum size=8mm,shape=circle,draw=black] (2) at (1,1) {2};
			\node[minimum size=8mm,shape=circle,draw=black] (3) at (-1,-1) {3};
			\node[minimum size=8mm,shape=circle,draw=black] (4) at (1,-1) {4};
			\path (1) edge (3);
			\path (1) edge (4);
			\path (2) edge (3);
			\path (2) edge (4);
			\path (3) edge (4);
		\end{tikzpicture}
		\caption{Example of a chordal graph on 4 alternatives with 5 edges}\label{fig:chordal_example2}
	\end{figure}
	Then, the reduced information matrix of such a design computes as
	\[
	M^{(4)}(\xi,\pi_0) =\begin{pmatrix}
		\lambda_{13} w_{13}+\lambda_{14} w_{14} & 0 & -\lambda_{13} w_{13} \\
		0 & \lambda_{23} w_{23}+\lambda_{24} w_{24} & -\lambda_{23} w_{23} \\
		-\lambda_{13} w_{13} & -\lambda_{23} w_{23} & \lambda_{13} w_{13}+\lambda_{23} w_{23}+\lambda_{34} w_{34} \\
	\end{pmatrix}.
	\]
	We obtain from Corollary~\ref{cor:KKT-conditions} that
	\begin{align*}
		\Gamma(\xi^*,\pi)&=\begin{pNiceMatrix}
			0&\Gamma_{12}&\frac{3}{\lambda_{13}}&\frac{3}{\lambda_{14}}\\
			\Gamma_{12}&0&\frac{3}{\lambda_{23}}&\frac{3}{\lambda_{24}}\\
			\frac{3}{\lambda_{13}}&\frac{3}{\lambda_{23}}&0&\frac{3}{\lambda_{34}}\\
			\frac{3}{\lambda_{14}}&\frac{3}{\lambda_{24}}&\frac{3}{\lambda_{34}}&0\\
		\end{pNiceMatrix},
	\end{align*}
	where $ \Gamma_{12} $ denotes the unknown entry in $ \Gamma(\xi^*,\pi) $.
	Via the Farris transform, we find
	\begin{align*}
		\Sigma^{(4)}(\xi^{*},\pi_0)=\begin{pNiceMatrix}
			\frac{3}{\lambda_{14}} & \sigma_{12} & \frac{3 (\lambda_{13} \lambda_{14}+\lambda_{13} \lambda_{34}-\lambda_{14} \lambda_{34})}{2 \lambda_{13} \lambda_{14} \lambda_{34}} \\
			\sigma_{12} & \frac{3}{\lambda_{24}} & \frac{3 (\lambda_{23} \lambda_{24}+\lambda_{23} \lambda_{34}-\lambda_{24} \lambda_{34})}{2 \lambda_{23} \lambda_{24} \lambda_{34}} \\
			\frac{3 (\lambda_{13} \lambda_{14}+\lambda_{13} \lambda_{34}-\lambda_{14} \lambda_{34})}{2 \lambda_{13} \lambda_{14} \lambda_{34}} & \frac{3 (\lambda_{23} \lambda_{24}+\lambda_{23} \lambda_{34}-\lambda_{24} \lambda_{34})}{2 \lambda_{23} \lambda_{24} \lambda_{34}} & \frac{3}{\lambda_{34}} \\
		\end{pNiceMatrix},
	\end{align*}
	where $ \sigma_{12}=\frac{1}{2}(\frac{3}{\lambda_{14}}+\frac{3}{\lambda_{24}}-\Gamma_{12}) $.
	According to \citet[Theorem~2.2.3]{BW2011}, there is only one positive definite solution for $ \sigma_{12} $ such that $ \left(\Sigma^{(4)}(\xi^{*},\pi_0)\right)^{-1}_{12}=0 $, namely
	\begin{align*}
		\sigma_{12}&=\frac{3 (\lambda_{13} \lambda_{14}+\lambda_{13} \lambda_{34}-\lambda_{14} \lambda_{34})}{2 \lambda_{13} \lambda_{14} \lambda_{34}} \frac{\lambda_{34}}{3}\frac{3 (\lambda_{23} \lambda_{24}+\lambda_{23} \lambda_{34}-\lambda_{24} \lambda_{34})}{2 \lambda_{23} \lambda_{24} \lambda_{34}}\\
		&=\frac{3 (-\lambda_{14} \lambda_{34}+\lambda_{13}(\lambda_{14}+\lambda_{34})) (-\lambda_{24}\lambda_{34}+\lambda_{23}(\lambda_{24}+\lambda_{34}))}{4 \lambda_{13}\lambda_{14}\lambda_{23}\lambda_{24}\lambda_{34}}.
	\end{align*}
	Via the inverse of $ \Sigma^{(4)}(\xi^{*},\pi_0) $ we obtain the $ D $-optimal information matrix $ M(\xi^*,\pi) $ where the entries are rational functions of $ \lambda_{uv} $ and hence the design weights $ w_{uv}^*=-\frac{M_{uv}(\xi^*,\pi)}{\lambda_{uv}} $.
	The design $ \xi^* $ is $ D $-optimal when $ \Gamma_{12}(\xi^*,\pi)\le \frac{3}{\lambda_{12}} $ and when $ \xi^*\in\Delta_{6} $.
\end{example}

\begin{example}
	Let $ m=4 $ and assume a design $ \xi $ supported on the chordal graph in Figure~\ref{fig:chordal_example}.
	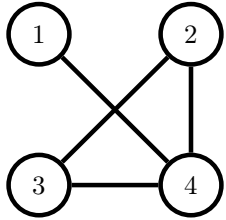
\begin{figure}[ht]\centering
		\begin{tikzpicture}[line width=.6mm]
			\node[minimum size=8mm,shape=circle,draw=black] (1) at (-1,1) {1};
			\node[minimum size=8mm,shape=circle,draw=black] (2) at (1,1) {2};
			\node[minimum size=8mm,shape=circle,draw=black] (3) at (-1,-1) {3};
			\node[minimum size=8mm,shape=circle,draw=black] (4) at (1,-1) {4};
			\path (1) edge (4);
			\path (2) edge (3);
			\path (2) edge (4);
			\path (3) edge (4);
		\end{tikzpicture}
		\caption{Example of a chordal graph on 4 alternatives with 4 edges}\label{fig:chordal_example}
	\end{figure}
	Then, the reduced information matrix $ M^{(4)}(\xi,\pi_0) $ computes as
	\[
	M^{(4)}(\xi,\pi_0)=\begin{pmatrix}
		\lambda_{14} w_{14} & 0 & 0 \\
		0 & \lambda_{23} w_{23}+\lambda_{24} w_{24} & -\lambda_{23} w_{23} \\
		0 & -\lambda_{23} w_{23} & \lambda_{23} w_{23}+\lambda_{34} w_{34} \\
	\end{pmatrix}.
	\]
	We obtain from Corollary~\ref{cor:KKT-conditions} that
	\begin{align*}
		\Gamma(\xi^*,\pi)&=\begin{pNiceMatrix}
			0&\Gamma_{12}&\Gamma_{13}&\frac{3}{\lambda_{14}}\\
			\Gamma_{12}&0&\frac{3}{\lambda_{23}}&\frac{3}{\lambda_{24}}\\
			\Gamma_{13}&\frac{3}{\lambda_{23}}&0&\frac{3}{\lambda_{34}}\\
			\frac{3}{\lambda_{14}}&\frac{3}{\lambda_{24}}&\frac{3}{\lambda_{34}}&0\\
		\end{pNiceMatrix},
	\end{align*}
	where $ \Gamma_{12},\Gamma_{13} $ denote the unknown entries of $ \Gamma(\xi^*,\pi) $.
	Via the Farris transform we obtain
	\begin{align*}
		\Sigma^{(4)}(\xi^{*},\pi_0)=\begin{pNiceMatrix}
			\frac{3}{\lambda_{14}} &\sigma_{12}  &\sigma_{13}  \\
			\sigma_{12} & \frac{3}{\lambda_{24}} & \frac{3 (\lambda_{23} \lambda_{24}+\lambda_{23} \lambda_{34}-\lambda_{24} \lambda_{34})}{2 \lambda_{23} \lambda_{24} \lambda_{34}} \\
			\sigma_{13} & \frac{3 (\lambda_{23} \lambda_{24}+\lambda_{23} \lambda_{34}-\lambda_{24} \lambda_{34})}{2 \lambda_{23} \lambda_{24} \lambda_{34}} & \frac{3}{\lambda_{34}} \\
		\end{pNiceMatrix},
	\end{align*}
	where $ \sigma_{12}=\frac{1}{2}(\frac{3}{\lambda_{14}}+\frac{3}{\lambda_{24}}-\Gamma_{12}) $ and $ \sigma_{13}=\frac{1}{2}(\frac{3}{\lambda_{14}}+\frac{3}{\lambda_{34}}-\Gamma_{13}) $.
	From the block structure of the information matrix we immediately obtain that the positive definite matrix completion is $ \sigma_{12}=\sigma_{13}=0 $.
	Via the inverse of $ \Sigma^{(4)}(\xi^{*},\pi_0) $ we obtain the $ D $-optimal information matrix $ M(\xi^*,\pi) $ where the entries are rational functions of $ \lambda_{uv} $ and hence the design weights $ w_{uv}^*=-\frac{M_{uv}(\xi^*,\pi)}{\lambda_{uv}} $.
	The design $ \xi^* $ is $ D $-optimal when 
	\begin{align*}
		\Gamma_{12}(\xi^*,\pi)&=\frac{3}{\lambda_{14}}+\frac{3}{\lambda_{24}}\le \frac{3}{\lambda_{12}},\\
		\Gamma_{13}(\xi^*,\pi)&=\frac{3}{\lambda_{14}}+\frac{3}{\lambda_{34}}\le \frac{3}{\lambda_{13}},
	\end{align*}
	and $ \xi^*\in\Delta_{6} $.
\end{example}

\subsection{Five alternatives}\label{sec:five_alternatives}
We only discuss the complete graph for the setting with five alternatives.
\begin{example}
	Similarly to Example~\ref{ex:complete_m=4}, we obtain from Corollary~\ref{cor:KKT-conditions} that $ D $-optimality of a fully-supported design holds when
	\begin{align*}
		\Gamma(\xi^*,\pi)=\begin{pNiceMatrix}
			0&\frac{4}{\lambda_{12}}&\frac{4}{\lambda_{13}}&\frac{4}{\lambda_{14}}&\frac{4}{\lambda_{15}}\\ 
			\frac{4}{\lambda_{12}}&0&\frac{4}{\lambda_{23}}&\frac{4}{\lambda_{24}}&\frac{4}{\lambda_{25}}\\
			\frac{4}{\lambda_{13}}&\frac{4}{\lambda_{23}}&0&\frac{4}{\lambda_{34}}&\frac{4}{\lambda_{35}}\\
			\frac{4}{\lambda_{14}}&\frac{4}{\lambda_{24}}&\frac{4}{\lambda_{34}}&0&\frac{4}{\lambda_{45}}\\    
			\frac{4}{\lambda_{15}}&\frac{4}{\lambda_{25}}&\frac{4}{\lambda_{35}}&\frac{4}{\lambda_{45}}&0\\       
		\end{pNiceMatrix},
	\end{align*}
	where  $ \lambda_{uv}= \frac{\pi_u \pi_v}{(\pi_u+\pi_v)^{2}} $.
	We obtain the $ D $-optimal information matrix as a rational function in $ \lambda_{uv},\;1\le u<v\le 5 $, such that  $ w_{uv}^*=-\frac{M_{uv}(\xi^*,\pi)}{\lambda_{uv}} $.
	The design is $ D $-optimal if and only if $ \xi^*\in\Delta_{10} $.
	The full computation is available in a \texttt{Mathematica} notebook.
\end{example}

\begin{example}
	Let $ m=4 $ and assume a design $ \xi $ supported on the chordal graph in Figure~\ref{fig:butterfly}.
	\begin{figure}[ht]\centering
		\begin{tikzpicture}[line width=.6mm]
			\node[minimum size=8mm,shape=circle,draw=black] (1) at (-2,1) {1};
			\node[minimum size=8mm,shape=circle,draw=black] (3) at (2,1) {3};
			\node[minimum size=8mm,shape=circle,draw=black] (2) at (-2,-1) {2};
			\node[minimum size=8mm,shape=circle,draw=black] (4) at (2,-1) {4};
			\node[minimum size=8mm,shape=circle,draw=black] (5) at (0,0) {5};
			\path (1) edge (2);
			\path (1) edge (5);
			\path (2) edge (5);
			\path (3) edge (4);
			\path (3) edge (5);
			\path (4) edge (5);
		\end{tikzpicture}
		\caption{Example of a block graph on 5 alternatives with 6 edges}\label{fig:butterfly}
	\end{figure}
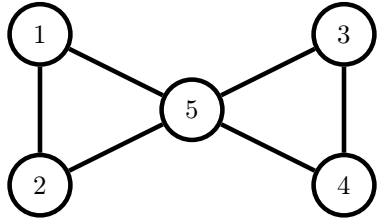
	Then, the reduced information matrix $ M^{(5)}(\xi,\pi_0) $ computes as
	\[
	M^{(5)}(\xi,\pi_0)=\begin{pmatrix}
		\lambda_{15} w_{15}+\lambda_{12} w_{12} & -\lambda_{12} w_{12} & 0&0 \\
		-\lambda_{12} w_{12} & \lambda_{12} w_{12}+\lambda_{25} w_{25} & 0&0 \\
		0 & 0 & \lambda_{34} w_{34}+\lambda_{35} w_{35}&-\lambda_{34} w_{34} \\
		0 & 0 & \lambda_{34} w_{34}&\lambda_{34} w_{34}+\lambda_{45} w_{45} \\
	\end{pmatrix}.
	\]
	We obtain from Corollary~\ref{cor:KKT-conditions} that
	\begin{align*}
		\Gamma(\xi^*,\pi)=\begin{pNiceMatrix}
			0&\frac{4}{\lambda_{12}}&\Gamma_{13}&\Gamma_{14}&\frac{4}{\lambda_{15}}\\ 
			\frac{4}{\lambda_{12}}&0&\Gamma_{23}&\Gamma_{24}&\frac{4}{\lambda_{25}}\\
			\Gamma_{13}&\Gamma_{23}&0&\frac{4}{\lambda_{34}}&\frac{4}{\lambda_{35}}\\
			\Gamma_{14}&\Gamma_{24}&\frac{4}{\lambda_{34}}&0&\frac{4}{\lambda_{45}}\\    
			\frac{4}{\lambda_{15}}&\frac{4}{\lambda_{25}}&\frac{4}{\lambda_{35}}&\frac{4}{\lambda_{45}}&0\\       
		\end{pNiceMatrix},
	\end{align*}
	From the simple block structure of $ \Sigma^{(5)}(\xi^*,\pi_0) $ we obtain the $ D $-optimal information matrix $ M(\xi^*,\pi) $ where the entries are rational functions of $ \lambda_{uv} $ and hence the design weights $ w_{uv}^*=-\frac{M_{uv}(\xi^*,\pi)}{\lambda_{uv}} $.
	The design $ \xi^* $ is $ D $-optimal when 
	\begin{align*}
		\Gamma_{13}&=\frac{4}{\lambda_{15}}+\frac{4}{\lambda_{35}}\le\frac{4}{\lambda_{13}},\\
		\Gamma_{14}&=\frac{4}{\lambda_{15}}+\frac{4}{\lambda_{45}}\le\frac{4}{\lambda_{14}},\\
		\Gamma_{23}&=\frac{4}{\lambda_{25}}+\frac{4}{\lambda_{35}}\le\frac{4}{\lambda_{23}},\\
		\Gamma_{24}&=\frac{4}{\lambda_{25}}+\frac{4}{\lambda_{45}}\le\frac{4}{\lambda_{24}}.
	\end{align*}
	
\end{example}

\end{document}